\documentclass[10pt]{amsart}
\usepackage{stmaryrd}
\usepackage{amssymb,amsmath,amscd,amsfonts,amsthm,dsfont,graphicx}
\usepackage[numbers]{natbib}
\usepackage{xspace}

\pdfoutput=1

\theoremstyle{plain}
\newtheorem{lemma}{Lemma}

\newtheorem{theorem}[lemma]{Theorem}

\theoremstyle{definition}

\newtheorem{remark}{Remark}

\newcommand{\pPoisson}{{$p$-Poisson}\xspace}

\newcommand{\pLaplacian}{{$p$-Laplacian}\xspace}

\newcommand{\qsine}{{$q$-sine}\xspace}
\newcommand{\qopt}{{q_{\text{opt}}}}

\newcommand{\be}{\begin{equation}}
\newcommand{\ee}{\end{equation}}
\newcommand{\N}{\mathbb{N}}

\newcommand{\R}{\mathbb{R}}

\newcommand{\Dt}{{\Delta t}}

\newcommand{\ud}{\,\mathrm{d}}

\newcommand{\lb}{\llbracket}
\newcommand{\rb}{\rrbracket}

\newcommand{\per}{\mathrm{per}}

\newcommand{\figref}[1]{{figure~\ref{#1}}}
\newcommand{\tabref}[1]{{table~\ref{#1}}}
\newcommand{\Figref}[1]{{Figure~\ref{#1}}}

\title[]{Approximation properties of the $q$-sine bases}
\author[]{Lyonell Boulton$^1$ and Gabriel Lord$^2$}
\address{Department of Mathematics and Maxwell Institute for Mathematical 
Sciences \linebreak Heriot-Watt University, Edinburgh
EH14 4AS, United Kingdom}
\email{$^1$L.Boulton@hw.ac.uk and $^2$G.J.Lord@hw.ac.uk}
\date{February 2011}

\begin{document}
\begin{abstract}
  For $q\geq \frac{12}{11}$ the eigenfunctions of the non-linear
  eigenvalue problem associated to the one-dimensional $q$-Laplacian
  are known to form a Riesz basis of $L^2(0,1)$.  We examine
  in this paper the approximation properties of this family of
  functions and its dual, and establish a non-orthogonal spectral
  method for the $p$-Poisson boundary value problem and its
  corresponding parabolic time evolution initial value problem with
stochastic forcing. The
  principal objective of our analysis is the determination of optimal
  values of $q$ for which the best approximation is achieved for a given $p$
  problem.
\end{abstract}
\maketitle

\tableofcontents
\pagebreak

\section{Introduction}
The $p$-\emph{Laplace operator} or $p$-\emph{Laplacian}, a
generalization of the ordinary Laplace operator, arises naturally in
applications from physics and engineering including: slow-fast
diffusion related to particles \cite{MR1419017}, superconductivity
\cite{Barrett2000977}, wavelet inpainting \cite{ZHANG2007546},
image processing \cite{Kuijper1} and game theory \cite{Rossi_TOW}. 
A typical application in the large
$p$ limit is a model for slow-fast diffusion for sandpiles
\cite{MR1451539}.  
Recently 
there has been a significant amount of research activity encompassing
methods of approximation for solution of non-linear partial
differential equations involving this operator
\cite{Barrett2000977,MR1276708,MR1192966}.  
The aim of the present paper is to further contribute to this activity by considering the particular case of the one-dimensional {$p$-Laplacian} and examine in detail the
approximation properties of a generalized spectral method described as follows.

Let $p>1$.  For $z\in \R$ let $\lb z\rb^{p-1}=z|z|^{p-2}$.  By
extension from the linear case corresponding to $p=2$, we define the
one-dimensional $p$-Laplacian to be the differential operator $
\Delta_p u=(\lb u' \rb^{p-1})' $. Here $u:[0,1]\longrightarrow \R$ is
such that $\lb u' \rb^{p-1}\in H^1(0,1)$.  The corresponding
$p$-Poisson boundary value problem is given by
\begin{equation} \label{ppoisson}
\begin{aligned}
&\Delta_p u(x)=g(x) & 0\leq x \leq 1 \\
&u(0)=u(1)=0
\end{aligned}
\end{equation}
where $g\in L^2(0,1)$. We also consider the related evolution equation 
\begin{equation} \label{evo_equ}
\begin{aligned}
&\frac{\partial u(x,t)}{\partial t}=[\Delta_p u(x,t)-g(x)]+ \nu \frac{\partial W(x,t)}{\partial t} \\
&u(x,0)=0 & 0\leq x \leq 1\\  &u(0,t)=u(1,t)=0 & t>0 \\
\end{aligned}
\end{equation} 
that includes a stochastic forcing term where the noise intensity
$\nu\geq 0$ and $W$ is a space-time Wiener process. Here  
\begin{equation}   \label{wiener}
  W(x,t) = \sum_{n \in \mathbb{N}} \beta_n(t) \psi_n(x)
\end{equation}
where $\beta_n$ are independent scalar Brownian motions, $\psi_n$ are
the eigenfunctions of the covariance operator $Q$ and $\alpha_n$ the
corresponding eigenvalues, see \cite{DaPZ,PrvtRcknr,Liu}. 
For the case of space-time white noise we
have $Q=I$. In the case of a deterministic system when $\nu=0$,
\eqref{ppoisson} is the steady state solution of \eqref{evo_equ}.

Let $q>1$. The so called $q$-\emph{sine functions} are defined as the
eigenfunctions of the $q$-Laplacian eigenvalue equation
\cite{MR680591,MR1469702,MR753635}:
\begin{equation}
\begin{aligned}
\label{sinep} 
-&\Delta_q u(x) =(q-1) \lambda \lb u(x)\rb ^{q-1} & 0\leq x\leq 1\\
&u(0)=u(1)=0.
\end{aligned}
\end{equation}
This family of functions and the corresponding problem \eqref{sinep}
was studied over 30~years ago by Elbert \cite{MR680591} and later by
\^Otani \cite{MR753635}, Bennewitz and Sait\={o} \cite{bensai} in connexion with the computation of optimal
constants in Sobolev-type inequalities. They generalize in a natural
fashion the 2-sine basis (corresponding to the linear case), and they
have very similar periodicity and interlacing structures.

In \cite{MR2240660} analogues of the classical completeness and
expansion theorems for the $q$-sine functions were established for
$q\geq 12/11$. Specifically it was shown that they form a Riesz basis
of $L^2(0,1)$. This leads to the following question: what are the
approximation properties of this basis, as well as its dual basis, and
how they relate to the approximation properties of the standard 2-sine
basis?

Below we address this question by examining approximation of the
solutions of \eqref{ppoisson} and \eqref{evo_equ} via projection
methods with a $q$-sine and a dual $q$-sine basis, regarding
$q>1$ and $p>1$ as free parameters.  A main focus of attention is
the determination of optimal values of $q$ for which the highest order
of convergence is achieved in a given $p$-problem. We 
demonstrate that standard properties of the $2$-sine basis applied to the
$p=2$ problem (such as super-polynomial convergence when $g(x)$ is
smooth) are lost, when a $q$-sine basis for $q\not=2$ is considered for
a $p\not=2$ problem.  As it turns out, the property of being a basis
for the $q$-sine functions
conceals a remarkably rich structure which is far from
evident given the apparent simplicity of problem \eqref{sinep}.

Background material on the $p$-sine basis and its dual is considered
in \S\ref{2}. There we examine a matrix representation of the Schauder
transform introduced in \cite{MR2240660}. This will be crucial in our
subsequent analysis as it gives rise to a stable procedure for
constructing numerically both bases.

In \S\ref{approx_est} we find estimates for the approximation of
square integrable functions in terms of their regularity. The dual
$q$-sine basis turns out to have very similar approximation properties
as the $q=2$ basis (lemma~\ref{dual}).  On the other hand, however, it
is fairly simple to construct smooth functions such that their
$q$-sine Fourier coefficients do not decay faster than a power $-5/2$
for $q>2$ (lemma~\ref{basis}).  The latter is in stark contrast with
the most elementary results in the numerical approximation of
solutions of differential equations by orthogonal spectral methods.

Section \ref{3} is devoted to the $p$-Poisson boundary value problem. In
theorem~\ref{stability} we find explicit uniform bounds on the
distance between any two solutions of \eqref{ppoisson}, given the
distance between the corresponding right hand sides.  We then examine
in detail the numerical computation of solutions of \eqref{ppoisson}
for source terms that are subject to various different regularity
constraints. As it turns out, the estimates established 
in Theorem~\ref{stability}
appear to be sub-optimal. A more thorough investigation in this respect will be
reported elsewhere. See \cite{ref1_2} for related results in the
context of finite element approximation of the solutions of \eqref{ppoisson},
including the higher dimensional case.

In the final \S\ref{5} we study the numerical approximation of 
solutions to \eqref{evo_equ} both in the deterministic and
stochastic systems.  We describe our discretization strategy
and solve this problem for different values of $p$ and $\nu$.  Our
results provide evidence on the performance of the $q$-sine basis for
the solution of  \eqref{evo_equ}, by showing the dependence
on the parameter $q$ of numerically computed $L^2$ residuals.

\section{The $q$-sine basis and its dual}
\label{2}

The $q$-Laplacian eigenvalue problem \eqref{sinep}, although
non-linear, has a fairly simple 
structure. The eigenvalues are found to be $\lambda=(n\pi_q)^q$ where
$\pi_q=\frac{2\pi}{q\sin(\pi / q)}$. 
The first eigenfunction $f_1(x)$ associated to the first
eigenvalue $(\pi_q)^q$ is strictly increasing in $[0,1/2]$, decreasing
in $[1/2,1]$ and it is even with respect to $x=1/2$. It can be
extended to an odd function (with respect to $x=0$) in the interval
$[-1,1]$ and then to a $2$-periodic $C^1$ function of
$\mathbb{R}$. If $q>2$ then $f_1''(x)$ is singular at $x=1/2$.
The eigenfunctions $f_n(x)$ associated to the eigenvalues $(n\pi_q)^q$
satisfy $f_n(x)=f_1(nx)$ for all $n\geq 2$.

In \figref{fig:basisQ} we have plotted the first three eigenfunctions
(top) and their corresponding derivatives (bottom) for (a) $q=1.4$ and
(b) $q=10$. They typify the case $q<2$ (a) and $q>2$ (b), respectively.
For large $q$ the basis functions $f_n$ approach zig-zag functions,
which are the eigenfunctions of the $\infty$-Laplace eigenvalue
problem.

Below we always assume that the family $\{f_n\}_{n\in\N}$ is
normalized by the condition $f'_n(0)=n\pi_q$ and leave implicit the
dependence of $f_n$ on $q$. In the special case $q=2$ we write $e_n(x)=\sqrt{2}\sin
(n\pi x)$, so that $\{e_n\}_{n\in\N}$ is an orthonormal basis.

\begin{figure}[hth]
\begin{center}
  (a) \hspace{0.48\textwidth} (b) \\
  \includegraphics*[width=0.48\textwidth]{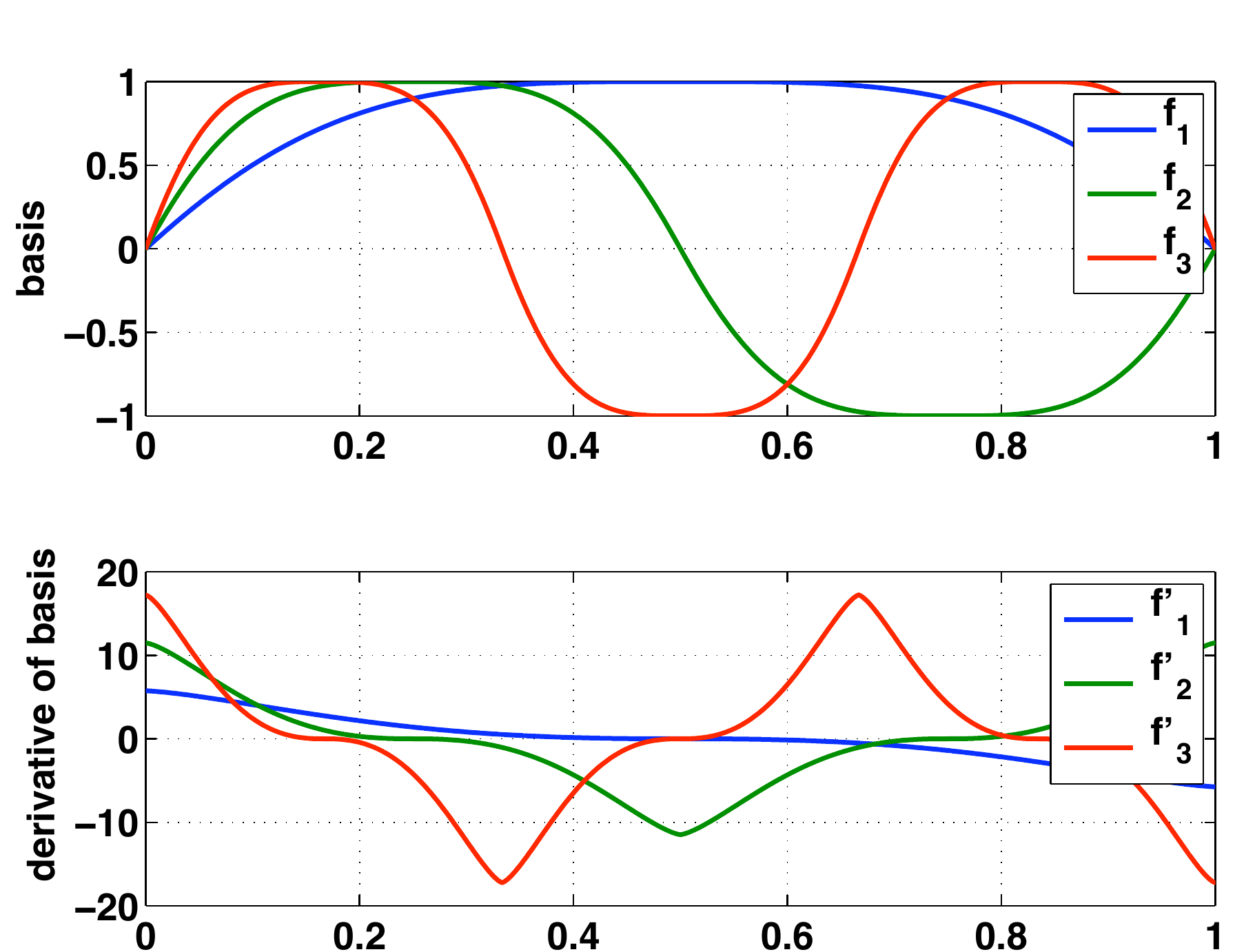}
  \includegraphics*[width=0.48\textwidth]{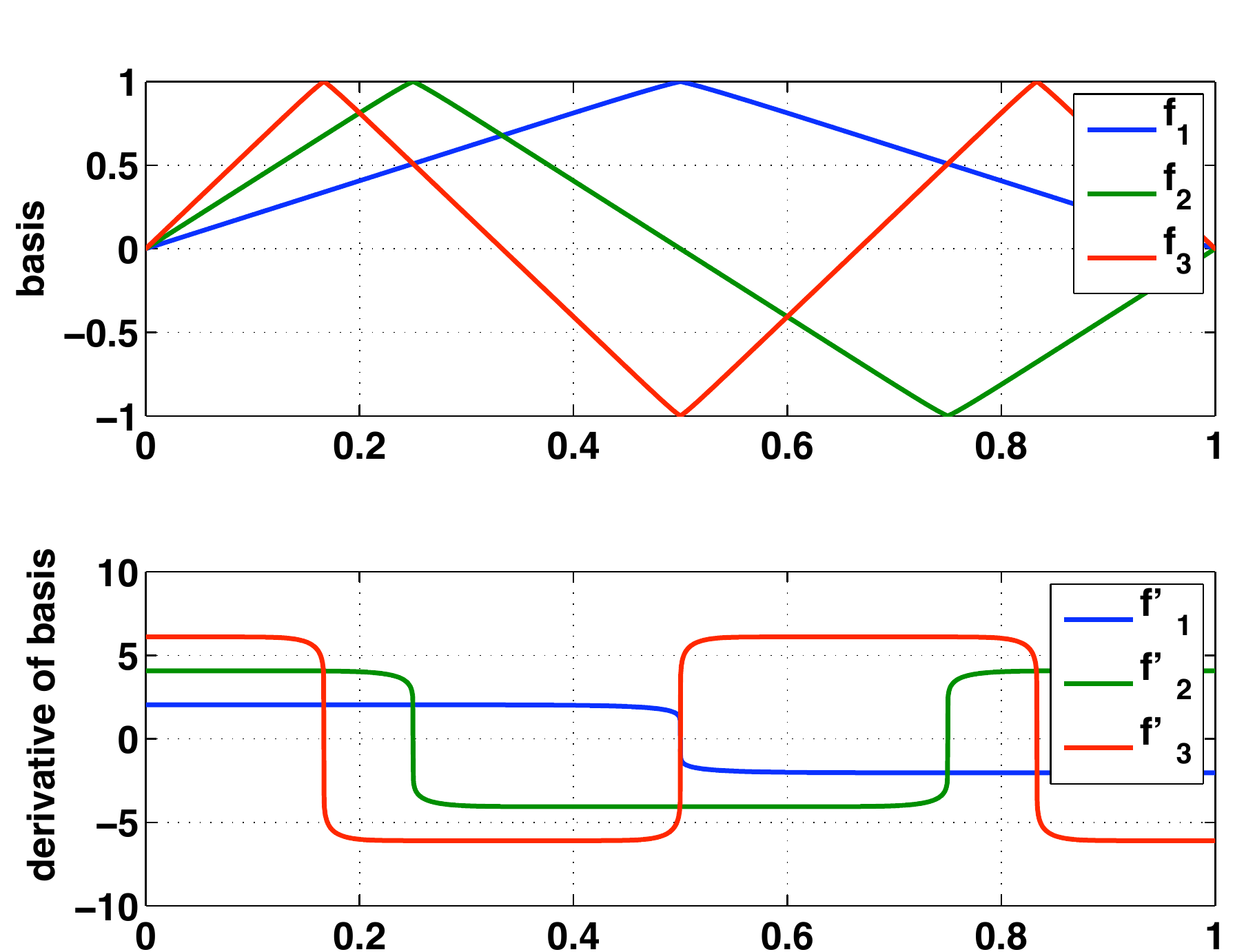}\\
  \caption{Approximation of the $q$-sine functions $f_j$ for $j=1,2,3$ (top)
    along with their derivatives (bottom) for (a) $q=1.4$ and (b) $q=10$.}
  \label{fig:basisQ}
\end{center}
\end{figure}

The Pythagorean identity generalizes to the $q$-sine functions 
\cite{MR680591} as
\begin{equation} \label{phyt}
   |f_1(x)|^q+\pi_q^{-q}|f_1'(x)|^q=1.
\end{equation}
Integrating this differential expression for small enough $x$ leads to the 
following explicit representation for the inverse function
\be
  f_1^{-1}(y) = \pi_q^{q}
  \int_0^y \frac{\ud x}{(1-x^q)^{1/q}} \qquad 0\leq y \leq 1.
  \label{eq:integral}
\ee 
As we will see below, this representation plays a crucial role
in the numerical estimation of $f_n(x)$.

Let the \emph{Schauder transform}, $T_q$, be the linear extension
of the mapping $e_n\longmapsto f_n$. Then $T_q:L^2(0,1)\longrightarrow
L^2(0,1)$ is an invertible bounded operator for all $q\geq 12/11$,
\cite[Theorem~1]{MR2240660}. Thus $\{f_n\}_{n\in\N}$ is a Riesz
basis of $L^2(0,1)$ for such range of the parameter $q$.  Further
evidence presented in \cite{MR2240660} suggests that in fact this is
also the case for all $q>1$, but at present this has not been proved
rigorously. Unless otherwise specified we will assume from now on that
$q\geq 12/11$.

The property of a Riesz basis ensures that every $g\in L^2(0,1)$ is
represented by a 
unique series expansion $ g=\sum_{n=1}^\infty a_n f_n$ which is
convergent in norm. The \emph{$q$-sine Fourier coefficients}, $a_n\in\R$, are
given explicitly by $a_n=\langle g,f^*_n\rangle$ where
$\{f^*_n\}_{n\in \N}$ is the basis dual to $\{f_n\}_{n\in \N}$. 
Since
$
   \delta_{jk}=\langle f_j,f_k^\ast\rangle=\langle T_qe_j,f_k^\ast\rangle=
    \langle e_j,T_q^\ast f_k^\ast\rangle
$
for all $j,k\in \N$, then $f^\ast_n=(T_q^{-1})^\ast e_n$.
It turns out that $f_n^*\not=f_n$ for $q\not=2$.

The following matrix representation of $T_q$ is fundamental to our
analysis. Let 
\[
   \tau_q(j)=\widehat{f_1}(j)=\sqrt{2}\int_0^1f_1(x)\sin(j\pi x) \ud x
\]
be the $j$th $2$-sine Fourier coefficient of $f_1(x)$. Then
the $k$th $2$-sine Fourier  coefficient of $f_n(x)$ is given by
\begin{equation} \label{entries_T}
\begin{aligned}
  \widehat{f_n}(k)&= \sqrt{2}\int_0^1f_1(nx)\sin(k\pi x) \ud x\\
  &= 2 \sum_{m\text{ odd}}\widehat{f_1}(m)\int_0^1 \sin(m\pi nx)
  \sin(k\pi x) \ud x \\
  &=\left\{ \begin{array}{ll} \tau_q(m) & \text{if }mn=k \text{ for
        some }m\text{ odd}\\ 0 & \text{otherwise.}\end{array} \right. 
\end{aligned}
\end{equation}
Hence $T_qe_n=\sum_{m=1}^\infty \tau_q(m)e_{mn}$ and therefore 
$T_q$ has a lower triangular matrix representation in the orthonormal
basis $\{e_n\}_{n\in\N}$. See figure~\ref{fig:dist_Tq}-(a).

\begin{figure}[hth]
\begin{center}
  (a) \hspace{0.48\textwidth} (b) \\
  \includegraphics*[width=0.52\textwidth]{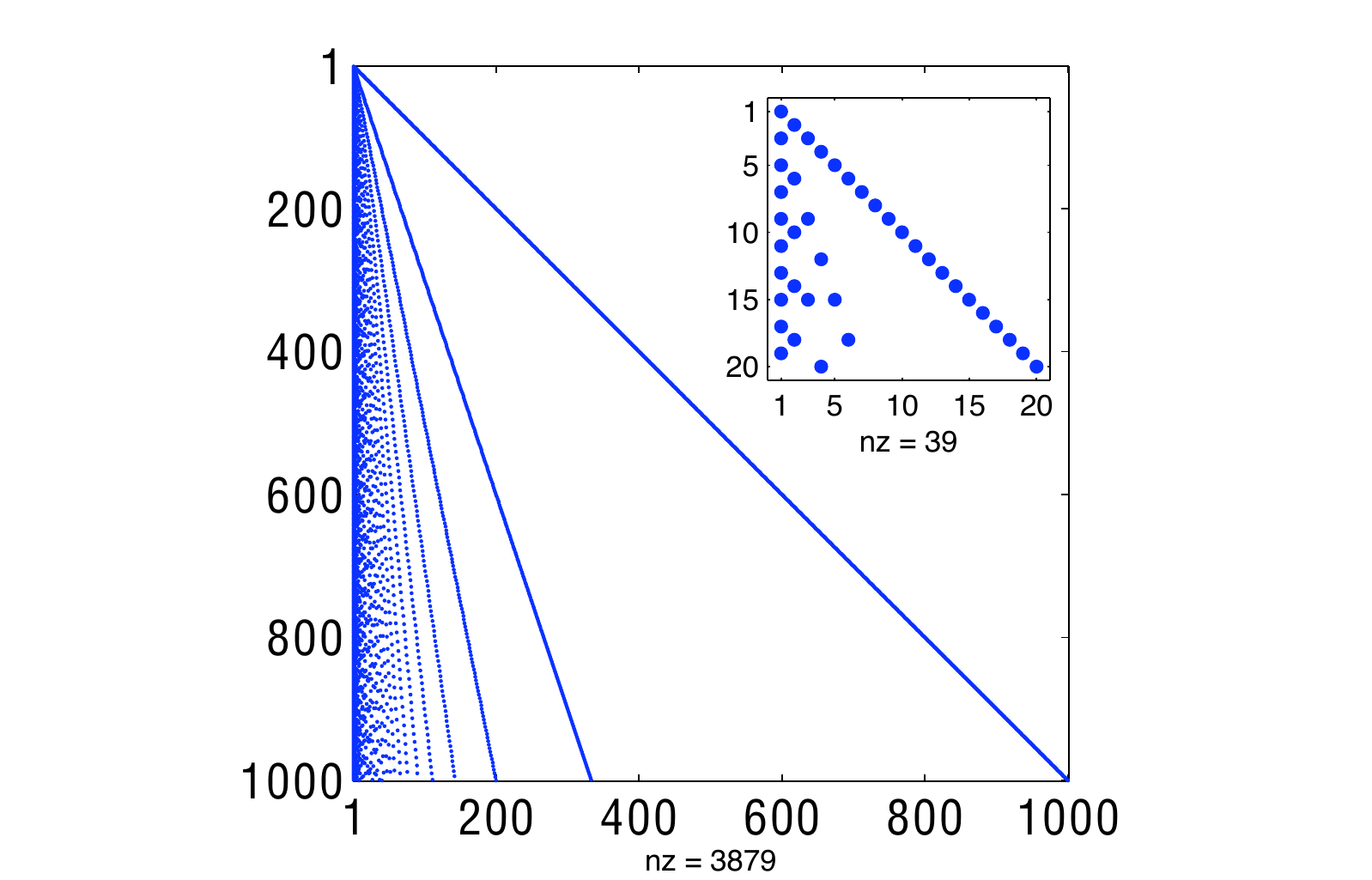}
  \includegraphics*[width=0.46\textwidth]{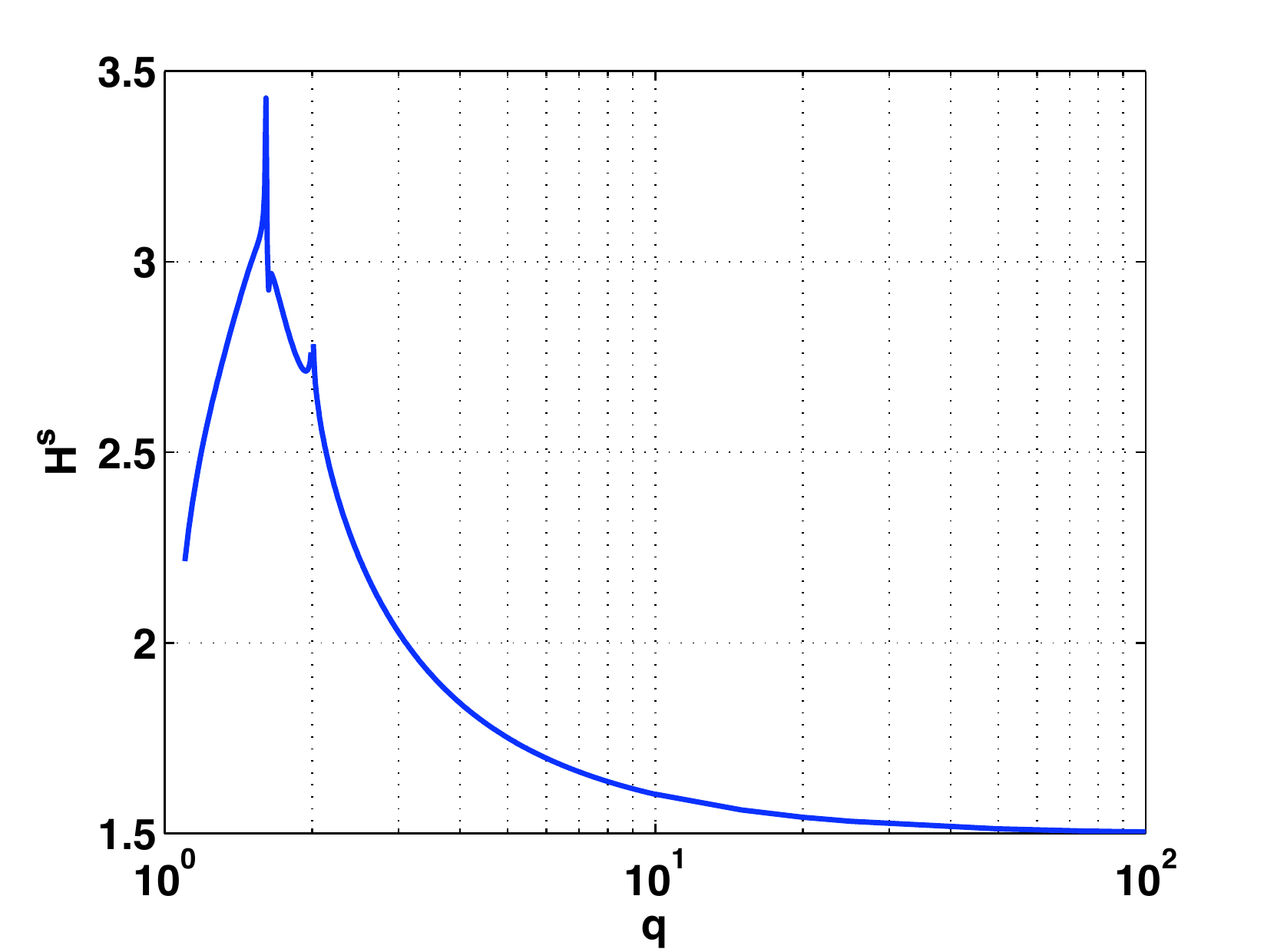}
  \caption{In (a) we plot the distribution of the non-zero entries 
    of a $1000\times 1000$ truncation of $T_q$. The insert corresponds to a
    $20\times 20$ truncation. The matrix entries are constant along
    each of the ``quasi-diagonals'' seen in the picture. 
    In (b) we plot $s$ against $q$ where $s$ is the numerically
    estimated  $H^s$ regularity of the basis function $f_1(x)$. Note
    that at $q=2$ the regularity is infinite.
\label{fig:dist_Tq}}
\end{center}
\end{figure}

\begin{remark} \label{q_infinity_case}
The basis of eigenvectors of the $\infty$-Laplace
eigenvalue problem are zig-zag functions,
\cite[Section~5]{MR2240660}. In this case we can write
$\tau_\infty(j)$ explicitly. As it turns out,
\[
     \lim_{q\to\infty}\tau_q(j)=(-1)^j\frac{8}{j^2\pi^2}=\tau_\infty(j).
\]
\end{remark}

Let $s>0$. Below we denote by $H^s_\per[0,1]$ the Sobolev space of
1-periodic functions $g\in L^2(0,1)$ such that $\sum_{j=1}^\infty
(1+j^2)^s|\widehat{g}(j)|^2<\infty$. 

\begin{lemma} \label{Sobolev_belongness} Let $f_1$ be the first
  $q$-sine function as defined above. If $1<q<2$, then $f_1\in
  H^{2}_\per[0,1]$. If $q>2$, then $f_1\in \bigcup_{s<3/2} H^s_\per[0,1]\setminus
  H^3_\per[0,1]$. If $q>4$, then $f_1\not\in H^2_\per[0,1]$.
\end{lemma}
\begin{proof}
A straightforward argument involving integration by parts yields
\[
   \tau_q(j)=-\frac{2\sqrt{2}}{j^2\pi^2}\int_0^{1/2} f_1''(x)\sin(j\pi x) 
   \ud x.
\]
Then
\begin{equation} \label{est_tau}
\begin{aligned}
|\tau_q(j)|&\leq \frac{2\sqrt{2}}{j^2\pi^2}\int_{0}^{1/2} |f_1''(x)| \ud x =
-\frac{2\sqrt{2}}{j^2\pi^2} \int_0^{1/2} f_1''(x) \ud x \\
& = -\frac{2\sqrt{2}}{j^2\pi^2}\left[f_1'(x)\right]_{0}^{1/2}=
\frac{2\sqrt{2}\pi_q}{j^2\pi^2}.
\end{aligned}
\end{equation}
Hence $\tau_q(j)=o(j^{-2})$ as $j\to
\infty$ and $f_1\in \bigcup_{s<3/2} H^s_\per[0,1]$.  From \eqref{phyt} it follows that $f''_1(x)=h(f_1(x))$ for
$h(y)=-\pi_q^2 y^{q-1}(1-y^q)^{\frac{2-q}{q}}$. 

Let $1<q<2$. Then
\[
    \int_0^{1/2} |f_1''(x)|^2 \ud x \leq \frac12 
\max_{y\in[0,1]} |h(y)|^2=\frac{\pi_q^4}{2}, 
\]
so $f_1\in  H^{2}_\per[0,1]$.

Let $q>2$. Since $\lim_{x\to \frac12^\pm} f_1''(x)=\mp \infty$, then
$f_1\not \in H^3_\per[0,1]$. Moreover, $f_1(x)\geq 2x$ for $0\leq x\leq \frac12$.
Then
\begin{align*}
   \int_0^{1/2} |f_1''(x)|^2 \ud x &\geq \int_0^{1/2} (2x)^{2(q-1)}
   (1-(2x)^{q})^{\frac{2(2-q)}{q}}\ud x. 
\end{align*}
If $q>4$, the integral on the right hand side diverges and so
$f_1\not\in H^2_\per[0,1]$.
\end{proof}

Evidently the $H^s$ regularity for $f_1(x)$ found in
lemma~\ref{Sobolev_belongness} is not optimal. 
{Figure~\ref{fig:dist_Tq}-(b)} shows a numerical estimation of the
precise value of $s(q)$, such that $f_1\in H^r_\per[0,1]$ for $r<s(q)$ and
$f_1\not\in H^r_\per[0,1]$ for $r>s(q)$. The data for this
graph was obtained by computing the decay rate of $\tau_q(j)$ for a
large truncation of $T_q$. A thorough investigation closely related to
this lemma in the higher dimensional context 
can be found in \cite{ref_3,ref1_1} and references therein.

In the large $q$ we have a limit of $s=3/2$ and this is confirmed by
Remark~\ref{q_infinity_case}.  At $q=2$ we simply have $f_1=\sin(\pi
x)$ and so $f_1$ is in $H^s_\per[0,1]$ for all $s$. According to lemma~\ref{Sobolev_belongness},
the curve should remain below $s=3$ as $q\to 2^+$. For $q>3$ the graph suggests $f_1\not\in
H^{2}_\per[0,1]$. However, if $q<3$ it suggests 
$s(q)>2$. As $q\to 1$ the regularity drops and the limit seems to
approach $s=2$. There is an interesting ``peak'' of
regularity around $q=1.6$ which we can not presently  explain.

The 2-sine Fourier coefficients of $f^*_n(x)$ are given by the $n$-th
columns of $(T_q^{-1})^{*}$.  This operator has an upper
triangular representation in the 2-sine basis. Then $f^*_n(x)$ are
trigonometric polynomials of order $n$. In fact,
$f^*_2(x)=\tau_q(1)^{-1}\sin(2\pi x)$ are parallel for all $q>1$.
Unlike the $q$-sine functions, not all dual $q$-sine functions have
the same periodicity structure.

We now describe a stable numerical procedure for computing these two
bases.  The \qsine functions can be approximated by first estimating
$f_1$ using \eqref{eq:integral}. Numerical integration yields
$f_1^{-1}$ on $[0,1/2]$. Although the integral is singular at $y=1$,
the value $f_1^{-1}(1)=1/2$ is known, so we do not need to consider
quadrature points too close to this singularity.  In our numerical
procedure we chose a fine uniform grid and apply a cumulative
Simpson's rule. This gives an approximation $\tilde{f}_1$ of $f_1$ for
$x\in [0,0.5]$ with a controlled tolerance and by symmetry we obtain $\tilde{f}_1$ for
$x\in[0,1]$. Note that $\tilde{f}_1$ is given on a non-uniform grid.
The Pythagorean identity \eqref{phyt} immediate yields the derivative
$f'_1$ of $f_1$ and hence we can use it to approximate the former with
a $ \widetilde{f_1'}$, defined also on the non-uniform grid.

Once we have constructed $\tilde{f}_1$ and $\widetilde{f_1'}$, we use
periodicity and symmetry to find corresponding approximations of $f_n$
and $f'_n$ for $n>1$. This involves considering scaled copies
$\tilde{f}_1$ and $\widetilde{f_1'}$, to form $\tilde{f}_n$ and
$\widetilde{f_n'}$. Here the number of non-uniform grid points on
$[0,1]$ grows with each $n=2:N$.

To obtain $f_n$ and $f_n'$ on a uniform grid $x_j=jh$ for $j=0:J$,
rather than the non-uniform grid that arises from the numerical
integration, we have considered numerical interpolation by piecewise
cubic polynomials. This gives $\{f^h_n\}_{n=1}^N$ and
$\{(f'_n)^h\}_{n=1}^N$ defined at $x=x_j$. The former is the
approximated basis and the latter the corresponding derivatives that
we use for further computation. 

Figure~\ref{fig:basisQ} was generated with an implementation of the
numerical scheme just described on an uniform grid with $J=4000$
points ($h=2.5\times 10^{-4}$). The integral \eqref{eq:integral} was
approximated with $2\times10^5$ points.

The dual basis is found from an $N\times N$ truncation, $T_q^N$, of
the Schauder transform.  In practice, we first compute
$\tau_q(j)$ and assemble $T^N_q$.  Then we define approximations
$(f^*_n)^h$ for $n=1:N$, as the trigonometric polynomials whose $k$th
2-sine Fourier coefficients are the $(n,k)$ entry of the matrix
$(T_q^N)^{-1}$.

\begin{remark}
In our numerical approximation of the set of basis functions we are
careful to fully resolve oscillations on the  basis
function $f_N$, taking at least $20$ mesh points per wavelength and for most
computations $100$ per wavelength. 
For $N\leq 50$ we use $h=(100N)^{-1}$ and so resolve each
oscillation in the basis function with $100$ spatial points. For
$N>50$ we use $h=(20N)^{-1}$ and so resolve with $20$ points.  
We also examined convergence of
orthogonality of the basis and dual in the spatial discretization and
noted $O(h^2)$ for $q\approx 10$ through to $O(h^4)$ for
$q\approx 1.4$. 
\end{remark}

\section{Approximation of source terms}
\label{approx_est}
Any given $g\in L^2(0,1)$ can be approximated by either  
\begin{gather*} 
g(x) \approx g^*_N(x)=\sum_{j=1}^N \langle g,f_j \rangle f_j^*(x)
\qquad \text{ or} \qquad
g(x) \approx g_N(x)=\sum_{j=1}^N \langle g,f_j^* \rangle f_j(x)
\end{gather*}
for large $N$. 
These two expansions converge as $N\to\infty$ in the
norm of $L^2(0,1)$ and also pointwise for almost all $x\in[0,1]$.
Unlike in the linear case corresponding to $q=2$, 
the rates of decrease of $\|g-g^*_N\|$ and $\|g-g_N\|$ 
can be very different when $q\not=2$.

Since the dual basis $\{f^*_n\}_{n\in\N}$ comprises trigonometric
polynomials, on the one hand we can formulate the following 
natural statement.

\begin{lemma} \label{dual} Let $g\in L^2(0,1)$. For all $N\in \N$,
\[
    \|g-g^*_{N}\| \leq \frac{\|T_q^{-1}\|\pi_q}{2\sqrt{2}} \Big\|g-\sum_{n=1}^N \widehat{g}(n) e_n \Big\|.
\]
\end{lemma} 
\begin{proof} By definition
$
    g-g^*_N=\sum_{n=N+1}^\infty \langle T_q^* g,e_n\rangle (T_q^*)^{-1}e_n.
$
Since the matrix associated to $T_q^*$ is upper triangular (see Section~\ref{2}), 
\[
    \langle T_q^*g,e_n\rangle =\sum_{k=n}^\infty [T_q^*]_{nk}\widehat{g}(k)=
    \sum_{k=1}^\infty \tau_q(k)\widehat{g}(nk).
\]
According to \eqref{est_tau}, we have 
$\sum_{k=1}^\infty |\tau_q(k)|\leq \frac{\pi_q}{2\sqrt{2}}$.
Hence
\begin{align*}
  |\langle T_q^*g,e_n\rangle|^2&=\Big|\sum_{k=1}^\infty \tau_q(k)\widehat{g}(nk)\Big|^2
  = \Big|\sum_{k=1}^\infty \tau_q(k)^{1/2}\tau_q(k)^{1/2}\widehat{g}(nk)\Big|^2 \\
  &\leq \Big(\sum_{k=1}^\infty |\tau_q(k)| \Big)
  \Big(\sum_{k=1}^\infty |\tau_q(k)||\widehat{g}(nk)|^2\Big)
  \leq \frac{\pi_q}{2\sqrt{2}} \sum_{k=1}^\infty |\tau_q(k)||\widehat{g}(nk)|^2.
\end{align*}
Thus,
\begin{align*}
   \|g-g_N^*\|^2&\leq \|T_q^{-1}\|^2\sum_{n=N+1}^\infty|\langle 
T_q^*g,e_n\rangle |^2 \leq
   \|T_q^{-1}\|^2\frac{\pi_q}{2\sqrt{2}} \sum_{n=N+1}^\infty \sum_{k=1}^\infty
   |\tau_q(k)||\widehat{g}(nk)|^2 \\
   &=\|T_q^{-1}\|^2\frac{\pi_q}{2\sqrt{2}} \sum_{k=1}^\infty \sum_{n=N+1}^\infty
   |\tau_q(k)||\widehat{g}(nk)|^2 
   = \|T_q^{-1}\|^2\frac{\pi_q}{2\sqrt{2}} \sum_{k=1}^\infty 
   |\tau_q(k)| \sum_{n=N+1}^\infty |\widehat{g}(nk)|^2 \\
   &\leq \|T_q^{-1}\|^2\frac{\pi_q}{2\sqrt{2}}  \sum_{k=1}^\infty|\tau_q(k)| 
   \sum_{n=N+1}^\infty |\widehat{g}(n)|^2
   = \|T_q^{-1}\|^2\frac{\pi_q^2}8  \sum_{n=N+1}^\infty |\widehat{g}(n)|^2 .
\end{align*}
\end{proof}

Therefore the $q$-sine dual expansion of any
$g\in C^\infty$ converges super-poly\-nomially fast.
On the other hand, however, it is not difficult to construct
examples of smooth functions $g$ with a subsequence of
 $q$-sine Fourier coefficients decaying slowly.

\begin{lemma} \label{basis}
Let $g(x)=e_1(x)$. If $j$ is prime, then
$
   \langle g,f_j^*\rangle=\tau_{q}(j) \tau_q(1)^{-1}.
$
\end{lemma}
\begin{proof}
  Since $T_q$ has an (infinite) lower triangular matrix representation
  in the orthonormal basis $\{e_n\}_{n\in \N}$, we can find the entries
  of the corresponding matrix representation of $T_q^{-1}$ by pivoting
  and forward substitution (Gaussian elimination). It is readily seen that $T_q^{-1}$ is
  necessarily lower triangular and its diagonal should be constant and
  equal to $\tau_q(1)^{-1}$. 

  Assume that $j$ is prime. According to \eqref{entries_T}, the only
  non-zero entries in the $j$th row of $T_q$ are $\tau_q(j)$ in the
  first position and $\tau_q(1)$ in the $j$th position. Therefore, the
  only non-zero entries in the $j$th row of $(T_q)^{-1}$ are
  $-\frac{\tau_q(j)}{\tau_q(1)}$ in the first position and $\tau_q(1)^{-1}$
  in the $j$th position. As the Fourier sine coefficients of $f_j^*$
  are obtained from the $j$th column of $(T_q^*)^{-1}$, the desired
  conclusion follows.
\end{proof}

By virtue of lemma~\ref{Sobolev_belongness}, the prime
$q$-sine Fourier coefficients of $\sin(\pi x)$ for $q>2$ can not
decrease faster than $j^{-5/2}$ in the large $j$ limit. Observe that
this is in stark contrast with the most elementary results in the
numerical approximation of solutions of differential equations by
orthogonal spectral methods.

\begin{remark} \label{using_gs} The finite set of basis functions
  $f_n$ and dual $f_n^*$ for $n=1:N$ generate corresponding $N$
  dimensional subspaces $V_N$ and $V_N^*$ of $L^2(0,1)$. Instead of
  computing directly with these non-orthogonal bases, one can apply
  the Gram-Schmidt algorithm in order to obtain orthonormal bases of
  these subspaces.  This has a numerical advantage of not needing to
  store both the basis and dual. We also considered this approach, however
  we found little advantage in terms of accuracy.
\end{remark}

Let us now consider various numerical tests on the approximation of
regular functions by $\{f_n\}_{n\in\N}$ and $\{f_n^*\}_{n\in\N}$.
Once the bases have been obtained on uniform mesh,
we can examine their approximation properties numerically by looking
at the decay of suitable residual for benchmark sources $g\in L^2(0,1)$.

\Figref{fig:res1} shows the typical outcomes of an experiment to
determine the dependence in $q$ of the $L^2(0,1)$ residual. We have
fixed here $N=40$, varied $q=1:100$, and computed residuals 
in the approximation of the following four functions: 
\begin{equation} 
\label{functions}
\begin{gathered}
g_\mathrm{a}(x)=f_{1}(x)+(2.5) f_{10}(x) \qquad \text{for }\quad q=10,\\
g_{\mathrm{b}}(x)=\left\{\begin{array}{ll} 1 & x\in [1/4,3/4] \\
0 & \text{otherwise} \end{array}\right. \\
g_{\mathrm{c}}(x)=\left\{\begin{array}{ll} (7/3)x  & x\in [0,3/7] \\
-(21/4)x+13/4& x\in[3/7,7/14] \\ (21/4)x-2& x\in[7/14,4/7] \\ 
-x+11/7 & x\in [4/7,1] \end{array} \right.\\
g_{\mathrm{d}}(x)=  (q-1)(k\pi_q)^q f_1|f_1|^{(q-2)} \qquad \text{for }\quad q=3,
\end{gathered}
\end{equation}
in (a)-(d) respectively. In the figure we include both the basis and
its dual, as well analogous calculations with the orthogonalized bases
of $V_{40}$ and $V_{40}^*$. In (a) we see an optimal $\qopt=10$ as
expected, in (b), (c) and (d) we increase the regularity of $g$ and
see $\qopt$ decrease with values  of $4.25,\,2.9,\,2.55$ for (b), (c) and
(d) respectively.

\begin{figure}[hth]
\begin{center}
  (a) \hspace{0.48\textwidth} (b)  \\
  \includegraphics*[width=0.48\textwidth]{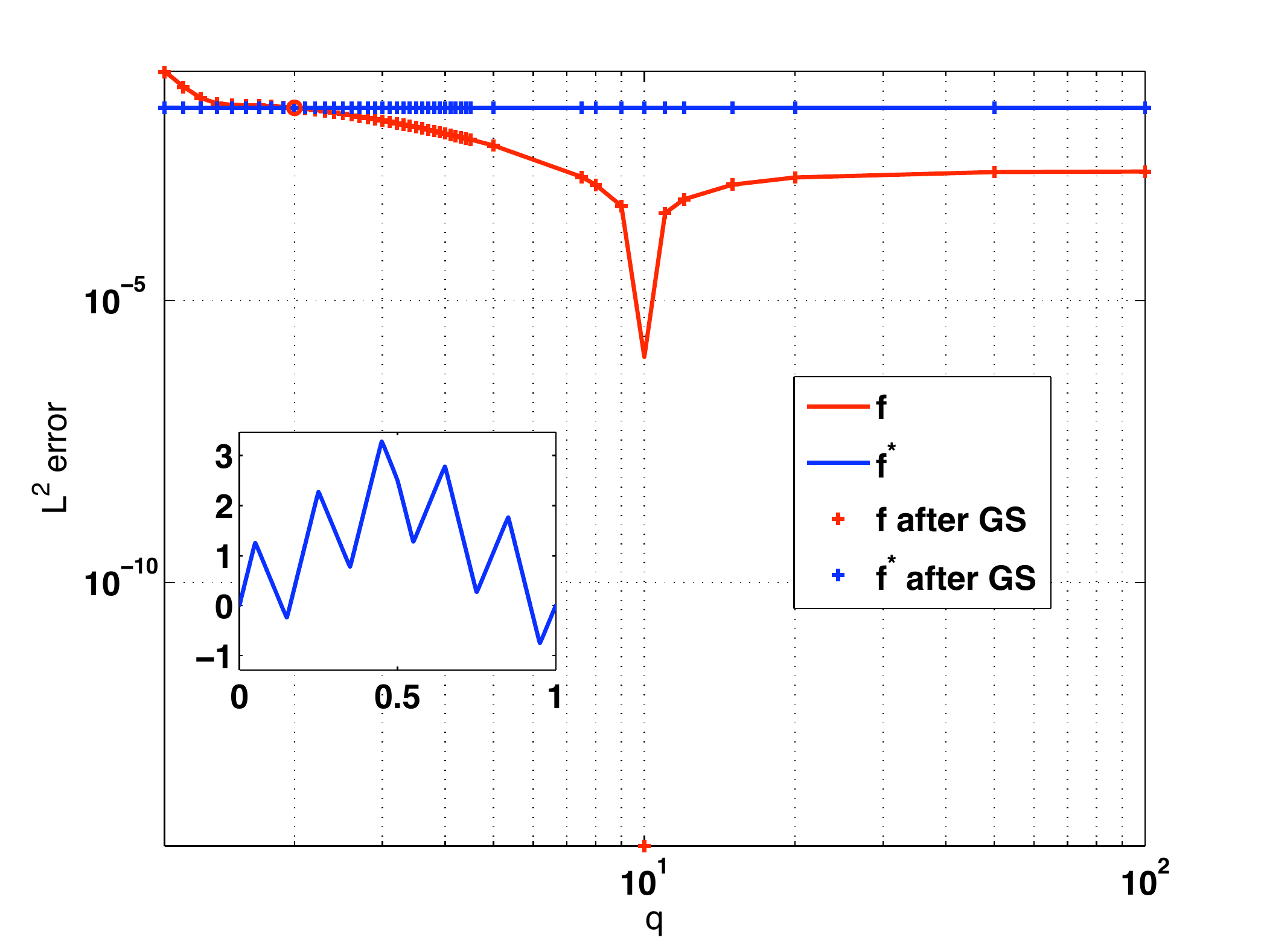}
  \includegraphics*[width=0.48\textwidth]{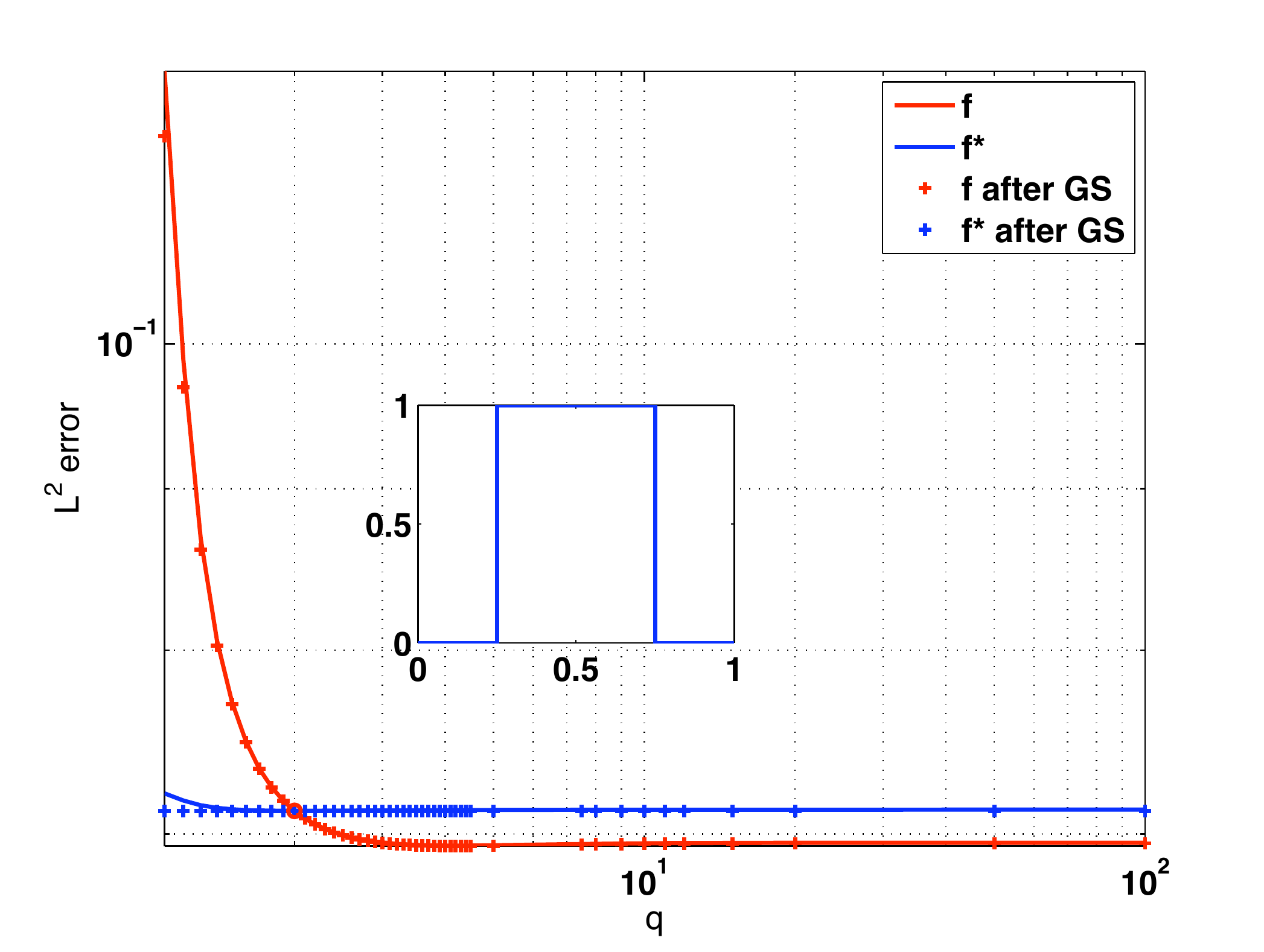}\\
  (c) \hspace{0.48\textwidth} (d)  \\
  \includegraphics*[width=0.48\textwidth]{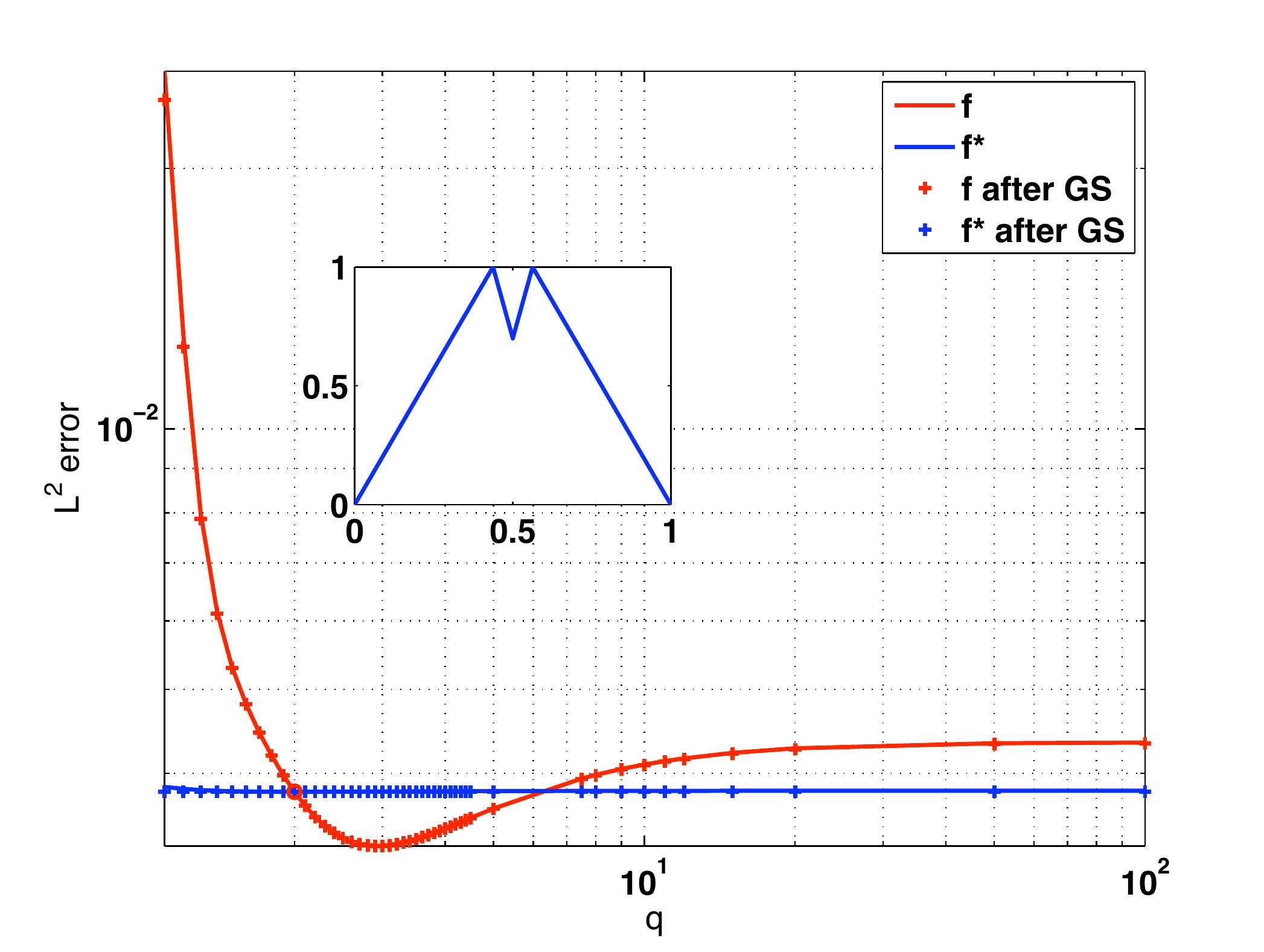}
  \includegraphics*[width=0.48\textwidth]{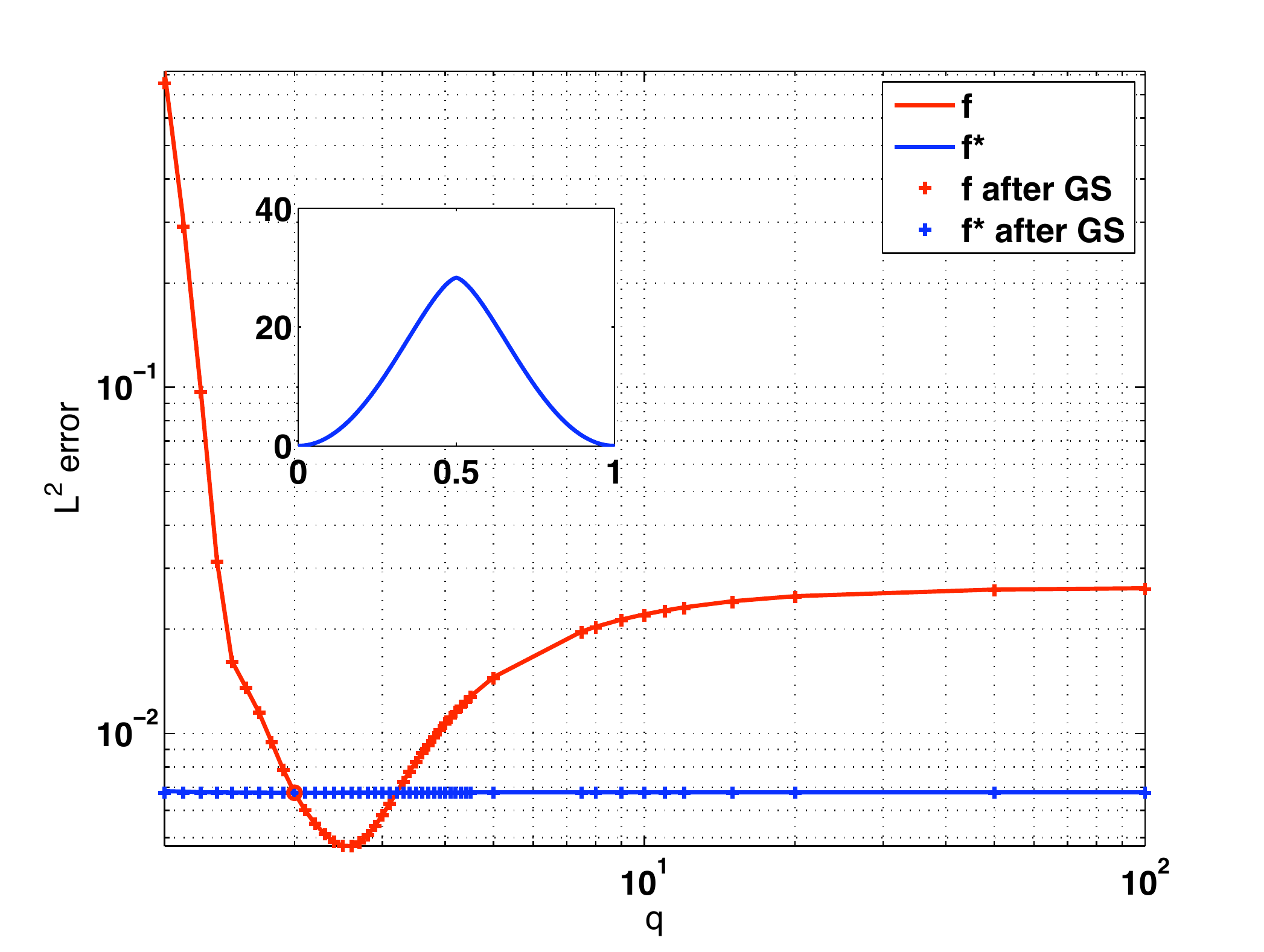}\\
  \caption{We vary a $q$-sine basis and dual basis and examine how the
    residual changes using $N=40$ modes for (a) $g$ a combination of
    two $10$-sine basis elements, (b) a piece-wise constant function
    $\qopt\approx 4.25$, (c) a piece-wise linear continuous
    $\qopt=2.9$ and (d) a differentiable with discontinuous derivative
    function $\qopt=2.55$. See 
    \eqref{functions}. \label{fig:res1} }
\end{center}
\end{figure}
This experiment gives a general insight about the $q$-behavior of
$L^2$-residuals in the approximation of functions with different
degrees of regularity by $q$-sine bases and their duals. For the
simple functions considered, our calculations indicate the following
general behavior of the $q$-sine basis:
\begin{itemize}
\item[(i)] as $q\to 1$ the residual deteriorates,
\item[(ii)] there is always a single minimum corresponding to an
  optimal $q=q_{\mathrm{opt}}$,
\item[(iii)] as $q\to \infty$ the residual curve becomes asymptotically
 constant with no local maximum for $q>2$.
\end{itemize}
In contrast, the approximation error is almost constant in $q$ for the
$q$-sine dual basis. This is indeed a consequence of Lemma~\ref{dual}
and the fact that $f_n^*$ are trigonometric polynomials.  Our tests
indicate that there does not seem to be a clear advantage in
orthogonalizing the basis or the dual basis for errors with an order
of magnitude above $10^{-5}$.

\begin{table}[hth]
\begin{tabular}{c||l|l|l|l|l}
$g(x)$ & $q=1.8$  & $q=2$ (exact)& $q=3$ & $q=5$ & $q=10$ \\\hline
$g_{\mathrm{b}}(x)$ & -0.4862  & -0.4905 (-0.5)  &-0.5005&-0.5052&-0.5069\\
$g_{\mathrm{c}}(x)$   & -1.4266  &   -1.4474 (-1.5)&-1.4959&-1.4842&-1.4438\\
$g_{\mathrm{d}}(x)$ & -1.9368  & -1.9952 (-2)    &-1.9826&-1.5934&-1.4595\\ 
$\sin(\pi x)$   &  -2.2778 & NaN  $(-\infty)$    &-1.9560&-1.6237&-1.4988\\
\end{tabular}
\caption{\label{tab:Nrate} Rates of convergence in $N$ 
for $g$ with different degrees of regularity and selected values of $q$.}
\end{table}

In \tabref{tab:Nrate} we have estimated $\alpha>0$ such that
$\|g-g_N\|<\beta N^{-\alpha}$ where $\beta>0$ is independent of $N$.
The data indicates that for $g\in H^1_\per[0,1]$, $q_{\mathrm{opt}}\approx
2$.  Moreover, as $q$ increases, we should expect $\alpha$ to decrease
and stabilize always below $1.5$ for $g(x)=\sin(\pi x)$.  This is
indeed suggested by lemma~\ref{Sobolev_belongness},
figure~\ref{fig:dist_Tq}-(b) and lemma~\ref{basis}, and it is
confirmed by the last row of the table.

\begin{figure}[hth]
  \begin{center}
    (a) \hspace{0.32\textwidth} (b) \hspace{0.32\textwidth} (c)  \\
    \includegraphics*[width=0.32\textwidth]{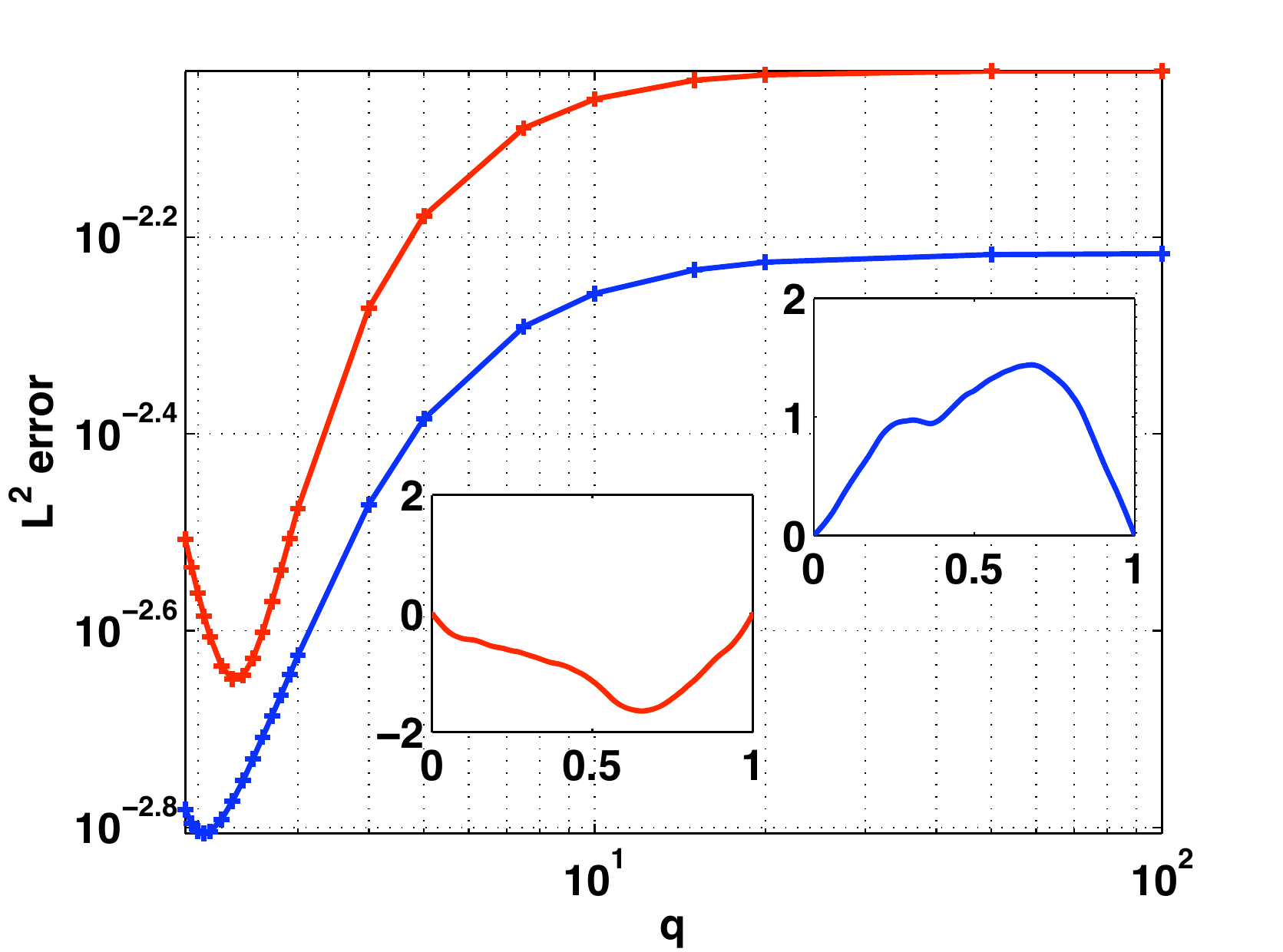}
    \includegraphics*[width=0.32\textwidth]{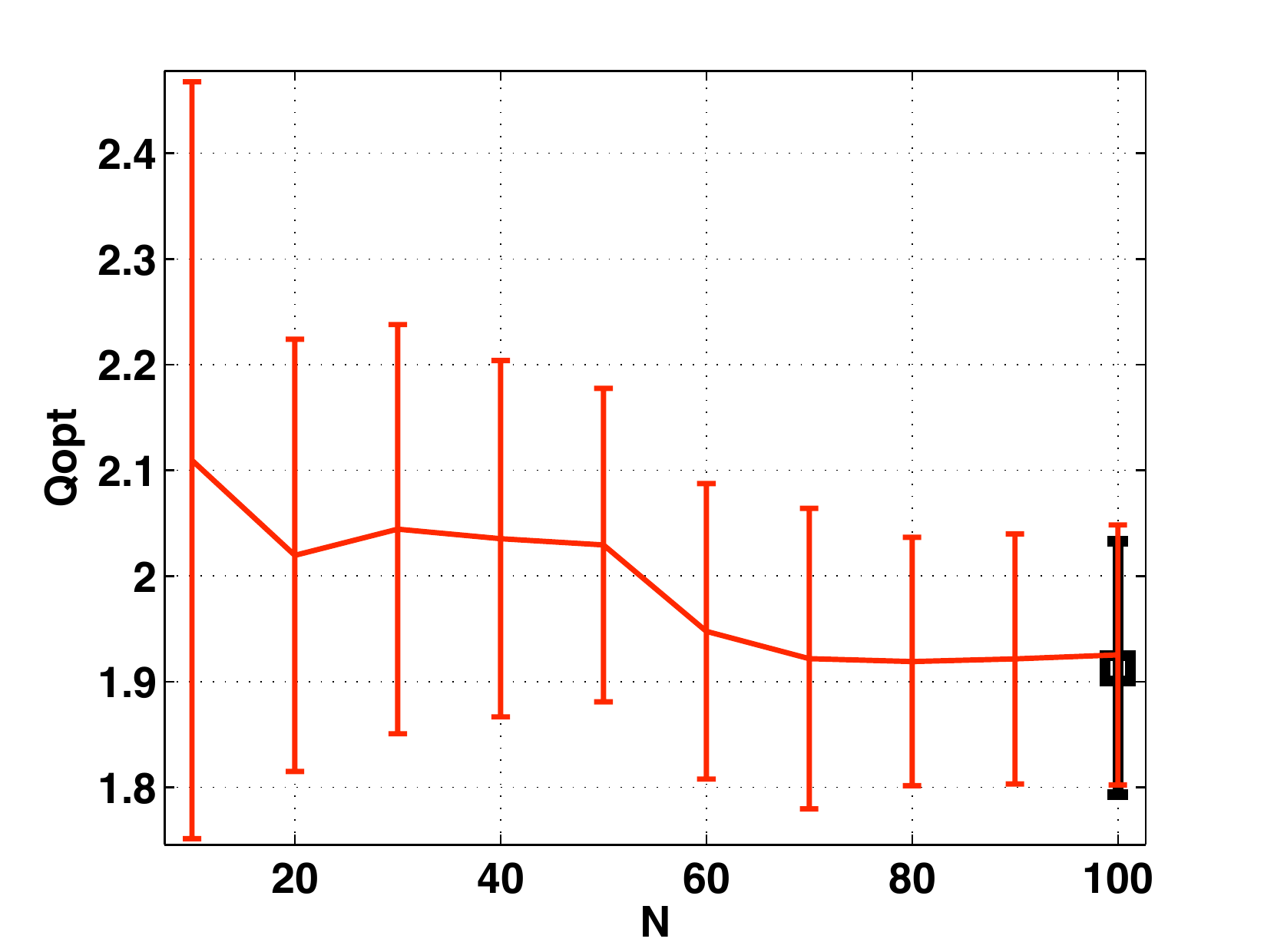}
    \includegraphics*[width=0.32\textwidth]{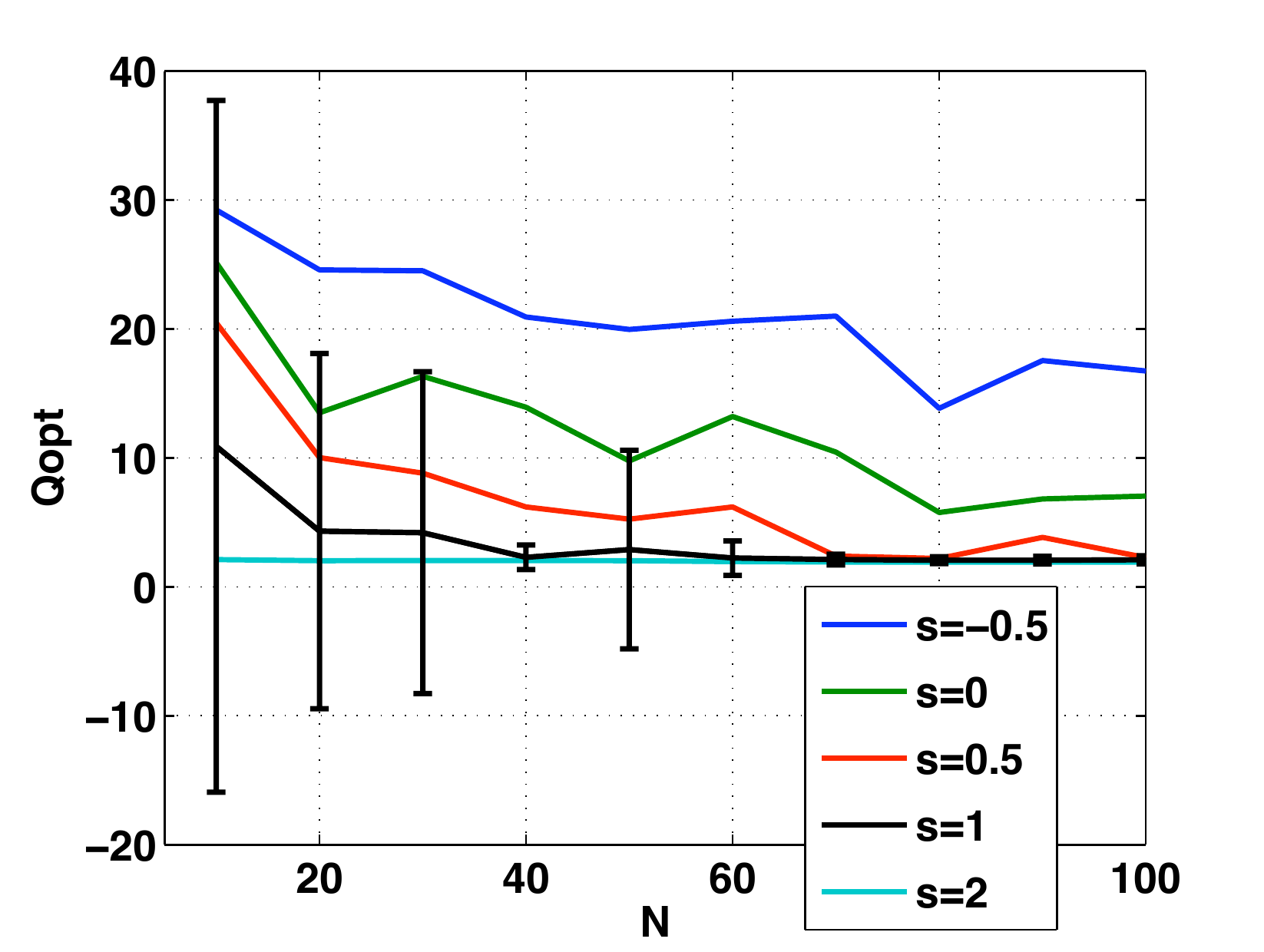}
    \caption{In (a) $L^2$ residual as a function 
      of $q$ where a $q$-sine basis is used to approximate 
      two different randomly generated $g\in H^2_\per[0,1]$ satisfying
      $g(0)=g(1)=0$. Note that the
      optimal $\qopt$ occurs at different values.
      In (b) and (c) we  examine $\qopt$ as $N$ increases. We include the mean
      and standard deviation over
      $200$ realizations of random functions subject to the same constraint. 
      For $N=100$ we also show the results from $1000$ realizations.}
    \label{fig:resrand}
  \end{center}
\end{figure}

In \figref{fig:resrand}-(a) we have computed the residual
$\|g-g_{10}\|$ (red) and $\|g-g_{20}\|$ (blue) for randomly generated
$g\in H^2_\per[0,1]$ constrained to $g(0)=g(1)=0$. To construct a random
function $g$ such that the $H^s_\per[0,1]$ norm is finite, 
we find $\widehat{g}(j) = a_j \beta_j$ where $a_j =
(1+j^2)^{-s/2}|j|^{-\frac12-\delta}$ for some small
$\delta>0$ and $\beta_j$ are independent identically distributed 
$N(0,1)$. 
Two realizations of functions $g$ obtained in this way are shown in
the inserts in \figref{fig:resrand}.  Note 
that $\qopt$ is achieved at different places but close to $2$.  In
\figref{fig:resrand}-(b) we examine $\qopt$ over $200$ realizations of
functions in $H^2_\per[0,1]$. For any fixed realization and value of $N$,
$\qopt\neq 2$ although in the limit this optimal parameter appears to
be close to 2.  For $N=100$ we also show the results of $1000$
realizations. In this case, the mean value is closer to $2$ although
the variation is still large.  In \figref{fig:resrand}-(c) we show the mean values where
functions are taken in $H^s_\per[0,1]$ for $s=-0.5,\, 0,\,0.5,\, 1$ and $2$
with $200$ realizations. For $s<2$ the variability in $\qopt$ is far
larger. For fixed $N$, $\qopt>2$.

The observed outcome of this experiment
strongly support the conjecture that $q_{\mathrm{opt}}$ typically
approaches $2$ as the regularity of $g$ is increased.

\section{Numerical solution of the $p$-Poisson equation}
\label{3}
We now address the question of approximating the solutions of
\eqref{ppoisson} by means of a $q$-sine and dual basis.  In view of
lemmas~\ref{dual} and \ref{basis}, we begin by determining uniform
estimates on how sensitive this solution is under perturbations of the
right hand side.  Analogous questions have certainly been considered
in more general frameworks, however here we focus on the explicit
calculation of the constants involved. 

A key ingredient in the estimates presented below 
is the fact that \eqref{ppoisson} can be integrated explicitly. Let the
Volterra operator
\[
     Vg(x)=\int_0^x g(t) \ud t.
\]
Note that $Vg\in H^1(0,1)$ for all $g\in L^1(0,1)$. Furthermore
$V:L^s(0,1)\longrightarrow L^t(0,1)$ is a contraction operator
for all $1\leq s,t \leq \infty$ and its norm can be explicitly determined
\cite[Theorem~1.1]{bensai}. Let
\[
     h_g(\gamma)=\int_0^1 \lb Vg(\tau) -\gamma \rb ^{\frac{1}{p-1}} \ud \tau.
\]
Then $h_g(\gamma)$ is a continuous function, decreasing in $\gamma$, for all
fixed $g\in L^1(0,1)$. Let 
\[
     \min_{x\in[0,1]}Vg(x)\leq \gamma_0(g) \leq \max_{x\in[0,1]}Vg(x)
\]
be the unique root such that $h_g(\gamma_0(g))=0$. Then
\begin{equation} \label{solu_ppoisson}
     u(x)=\int_0^x\lb Vg(\tau)-\gamma_0(g) \rb ^{\frac{1}{p-1}} \ud \tau
      =V\left(\lb Vg(\tau)-\gamma_0(g) \rb ^{\frac{1}{p-1}}\right)(x)
\end{equation}
is the unique solution of \eqref{ppoisson}.

We firstly establish concrete H\"older estimates on $h_g(\gamma)$. 
Without further mention in this section we will fix $r=\frac{1}{p-1}$ 
and denote by $\|\!\cdot\!\|_s$ the norm of $L^s(0,1)$. In the case $s=2$ 
we will continue suppressing the sub-index.

\begin{lemma} \label{cases} 
Let $g\in L^1(0,1)$ and $-\|g\|_1\leq \gamma 
\leq \mu \leq \|g\|_1$. Then
\begin{align*}
   (\mu-\gamma)& \leq
    \frac{2^{1-r}\|g\|_1^{1-r}}{r}
   [h_g(\gamma)-h_g(\mu)] & 0<r\leq 1 \\
 (\mu-\gamma)^r& \leq 2^{r-1}[h_g(\gamma)-h_g(\mu)] & r>1. 
\end{align*}
\end{lemma}
\begin{proof}
Suppose first that $0<r\leq 1$. From the graph of $\lb z \rb^r$
for $|z|\leq M$ it is readily seen that
$
       \lb z \rb^r - \lb w \rb ^r \geq r M^{r-1}(z-w)
$
for all $-M\leq w \leq z \leq M$. Then
\begin{align*}
   (\mu-\gamma) & = \int_0^1 \left[ (Vg(\tau)-\gamma)-(Vg(\tau)
-\mu)\right]\ud \tau \\ & \leq \frac{M^{1-r}}{r} \left[
h_g(\gamma)- h_g(\mu) \right]
\end{align*}
for $M=2\|g\|_1$, and $\gamma$ and $\mu$ as in the hypothesis. 

In a similar fashion, let $r>1$. Then
$
   (z-w)^r \leq 2^{r-1}(\lb z \rb^r - \lb w \rb ^r)
$
for all $-M\leq w \leq z \leq M$. 
Indeed, if $0\leq w <z$, a straightforward argument shows that
\[
  \frac{\lb z\rb^r-\lb w \rb^r}{z-w}=\frac{z^r-w^r}{z-w}\geq z^{r-1}\geq (z-w)^{r-1};
\]
if  $w<z\leq 0$,
\[
   \frac{\lb z\rb^r-\lb w \rb^r}{z-w}\geq |w|^{r-1}\geq (z-w)^{r-1};
\]
and if $w<0<z$,
\[
    \min_{z>0}\frac{\lb z\rb^r-\lb w \rb^r}{(z-w)^r} =
    \min_{z>0}\frac{z^r-|w|^r}{(z+|w|)^r}=\frac{1}{2^{r-1}}
\]
achieved when $w=-z$. Thus
\begin{align*}
   (\mu-\gamma)^r & = \int_0^1 \left[ (Vg(\tau)-\gamma)-(Vg(\tau)
-\mu)\right]^r\ud \tau \\ & \leq 2^{r-1} \left[
h_g(\gamma)- h_g(\mu) \right].
\end{align*}
\end{proof}

\begin{theorem} \label{stability}
Let $u$ and $\tilde{u}$ be solutions
of \eqref{ppoisson} with corresponding sources $g$ and 
$\tilde{g}$.
Let $m=\max\{\|g\|_1,\|\tilde{g}\|_1\}$. Then
\begin{align*}
 \|u-\tilde{u}\|&\leq 2^{1-r} \left(\|g-\tilde{g}\|+\frac{(4m)^{1-r}}{r} 
\|g-\tilde{g}\|_1^r\right)^{r} & 0<r\leq 1
 \\
\|u-\tilde{u}\|&\leq r2^{r-1}m^{r-1}\left(\|g-\tilde{g}\|+
2^{2-2/r}m^{1-1/r}r^{1/r}\|g-\tilde{g}\|_1^{1/r}\right) & r>1.
\end{align*}
\end{theorem}
\begin{proof}
Let $0<r\leq 1$ and $s=2/r\geq 2$. By virtue of \eqref{solu_ppoisson}
\[
   \|u-\tilde{u}\|\leq\big\|V\left(\lb Vg -\gamma_0(g) \rb^{r}-
    \lb V\tilde{g} -\gamma_0(\tilde{g}) \rb^{r}\right) \big\|_{s}\leq \big\|\lb Vg -\gamma_0(g) \rb^{r}-
    \lb V\tilde{g} -\gamma_0(\tilde{g}) \rb^{r} \big\|_{s}.
\]
Note that $|\lb z \rb ^r - \lb w\rb^r|\leq 2^{1-r} |z-w|^r$ 
for all $0\leq |w|\leq |z|$. Thus
\begin{align*}
 \big\|\lb Vg -\gamma_0(g) \rb^{r}-
    \lb V\tilde{g} -\gamma_0(\tilde{g}) \rb^{r} \big\|_{s}^{1/r} &
\leq 2^{\frac{1-r}{r}}\left( \|V(g-\tilde{g})\| + |\gamma_0(\tilde{g})-
\gamma_0(g)|\right) \\
& \leq 2^{\frac{1-r}{r}}\left( \|g-\tilde{g}\| + |\gamma_0(\tilde{g})-
\gamma_0(g)|\right).
\end{align*}
According to lemma~\ref{cases},
\begin{align*}
    |\gamma_0(\tilde{g})-\gamma_0(g)|
    &\leq \frac{(2m)^{1-r}}{r} |h_g(\gamma_0(\tilde{g}))-h_g(\gamma_0(g))| 
    = \frac{(2m)^{1-r}}{r} |h_g(\gamma_0(\tilde{g}))-h_{\tilde{g}}(\gamma_0(\tilde{g}))| \\
    &\leq \frac{(2m)^{1-r}}{r} \int_0^1 \Big|\lb Vg - \gamma_0(\tilde{g}) \rb^r-
                     \lb V\tilde{g} - \gamma_0(\tilde{g}) \rb^r \Big| 
\ud \tau \\
    & \leq \frac{(4m)^{1-r}}{r} \|g-\tilde{g}\|_r^r 
     \leq \frac{(4m)^{1-r}}{r} \|g-\tilde{g}\|_1^r.    
\end{align*}
This ensures the first statement.

Let $r> 1$. In a similar fashion as before, we see that
\begin{align*}
     \|u-\tilde{u}\|&\leq\big\|\lb Vg -\gamma_0(g) \rb^{r}-
    \lb V\tilde{g} -\gamma_0(\tilde{g}) \rb^{r} \big\| \\ &\leq
   r2^{r-1}m^{r-1} \left( \|V(g-\tilde{g})\|+|\gamma_0(g)-\gamma_0(\tilde{g})
|\right) \\
& \leq  r2^{r-1}m^{r-1} \left( \|g-\tilde{g}\|+|\gamma_0(g)-
\gamma_0(\tilde{g})|\right).
\end{align*}
Lemma~\ref{cases} and similar arguments as for the previous case, yield 
\begin{align*}
    |\gamma_0(\tilde{g})-\gamma_0(g)|^r
    &\leq 2^{r-1}\left|h_{\tilde{g}}(\gamma_0(\tilde{g}))- 
    h_g(\gamma_0(\tilde{g}))\right| \\
  & \leq  2^{2r-2}m^{r-1}r \|g-\tilde{g}\|_1.
\end{align*}
This completes the proof.
\end{proof}

The right hand side bound in the above theorem approaches 2 as
$r\to 0$ independently of the value of $\|g-\tilde{g}\|$. In fact
$\|u-\tilde{u}\|\leq \left\| l_r \right\|_1$ for $l_r=\lb Vg
-\gamma_0(g) \rb^{r}-\lb V\tilde{g} -\gamma_0(\tilde{g}) \rb^{r}$.
Since $l_r(\tau)\to 0$ for almost all $\tau\in[0,1]$ and
$|l_r(\tau)|\leq 2+\|Vg\|_\infty+\|V\tilde{g}\|_\infty$, the
Dominated Convergence Theorem yields $\|u-\tilde{u}\|\to 0$ as $r\to 0$.
This is a well-known property of the $\infty$-Laplacian, see \cite{Rossi_TOW}.
As mentioned in the introduction, 
the approach considered in \cite{ref1_2}  in the
context of finite element approximation of the solutions of \eqref{ppoisson},
may provide an insight on whether the constants found in the above theorem
are optimal.

\begin{figure}[hth]
\begin{center}
  (a) \hspace{0.48\textwidth} (b)  \\
  \includegraphics*[width=0.48\textwidth]{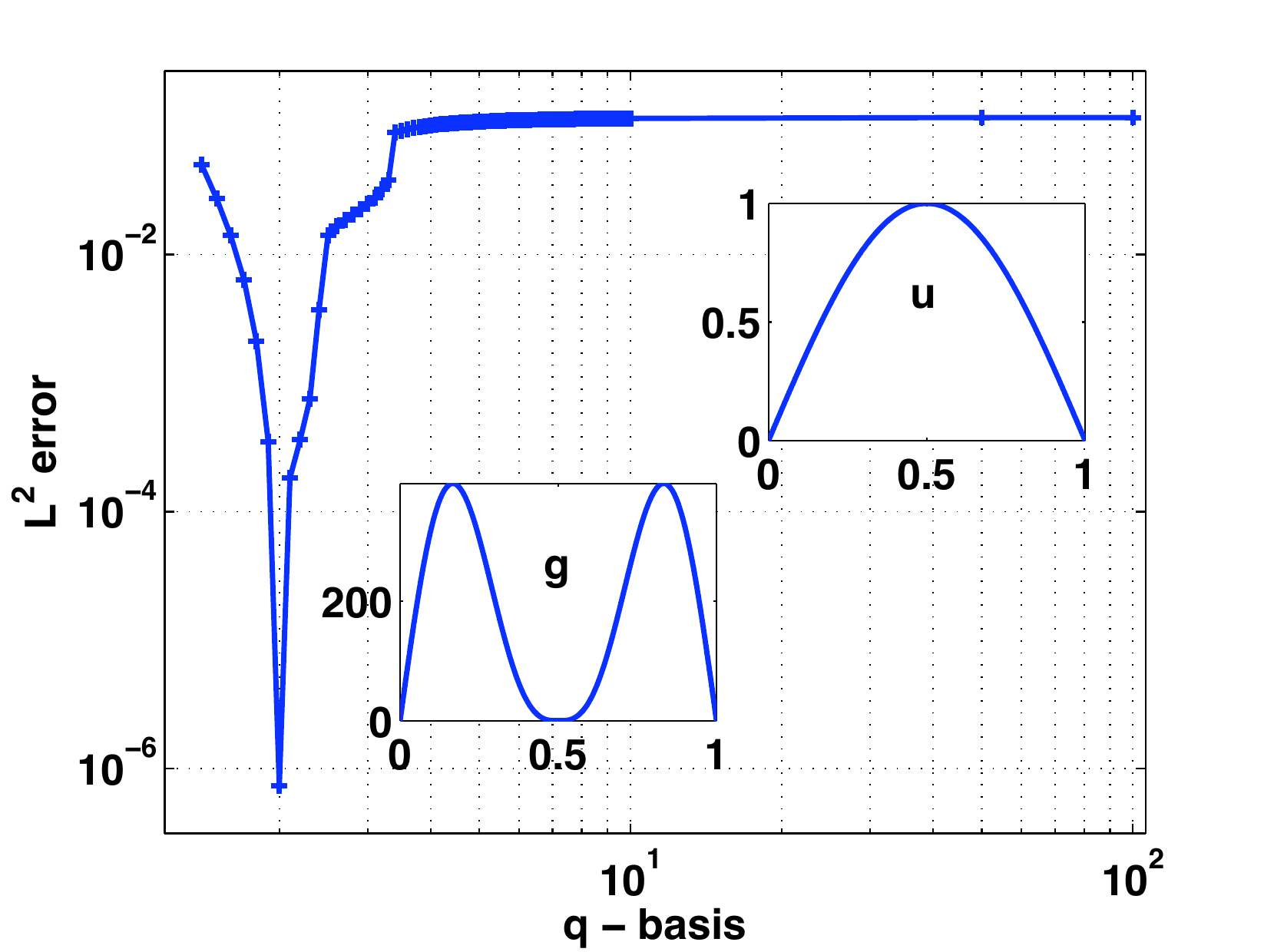}
  \includegraphics*[width=0.48\textwidth]{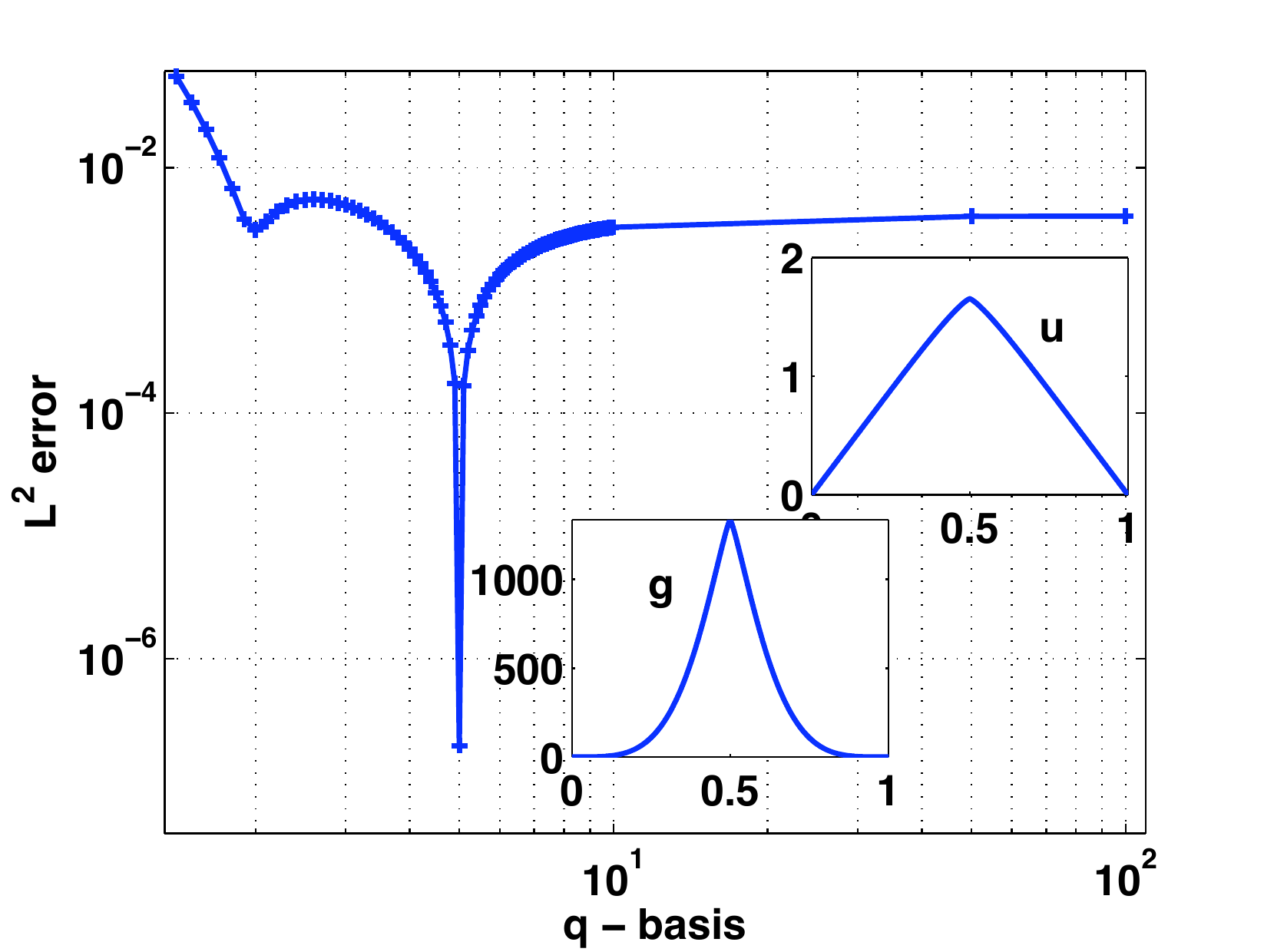}
  \caption{Solving the \pLaplacian problem with $p=5$. (a) The most
    accurate basis for a solution $u(x)=\sin(\pi x)$ is the standard 2-sine
    basis. (b) However for a solution $u(x)=f_1(x)$ with $q=5$ the 5-sine
    basis is the most accurate.  }
  \label{fig:plap1}
\end{center}
\end{figure}

We now describe how the \pPoisson problem is
discretized. The strong formulation \eqref{ppoisson} leads to the following
weak formulation using integration by parts and the boundary conditions:
\be 
\label{eq:weak}
\int_0^1 |u'|^{p-2}u'v' \ud x = \int_0^1 g v \ud x
\ee
for any absolutely continuous test function $v$.
We expand both $u$ and $v$ using a basis $\{\phi_n\}_{n\in\N}$, where $\phi_n$ is
either $f_n$ or $f_n^*$.  
After truncation this leads to the following nonlinear system of equations 
for unknown coefficients $c_j = \langle u,\phi_j^* \rangle$:
\be
\sum_{k=1}^N c_k \int_0^1 \left| \sum_{\ell=1}^N c_\ell \phi_\ell'\right|^{p-2}\phi_k'\phi_j' \ud x = \langle g, \phi_j \rangle \qquad j=1:N.
\label{eq:discrete}
\ee
The case $p=2$ evidently reduces to the standard linear system to solve for the
Poisson equation. 

\begin{figure}[hth]
\begin{center}
  (a) \hspace{0.48\textwidth} (b)  \\
  \includegraphics*[width=0.48\textwidth]{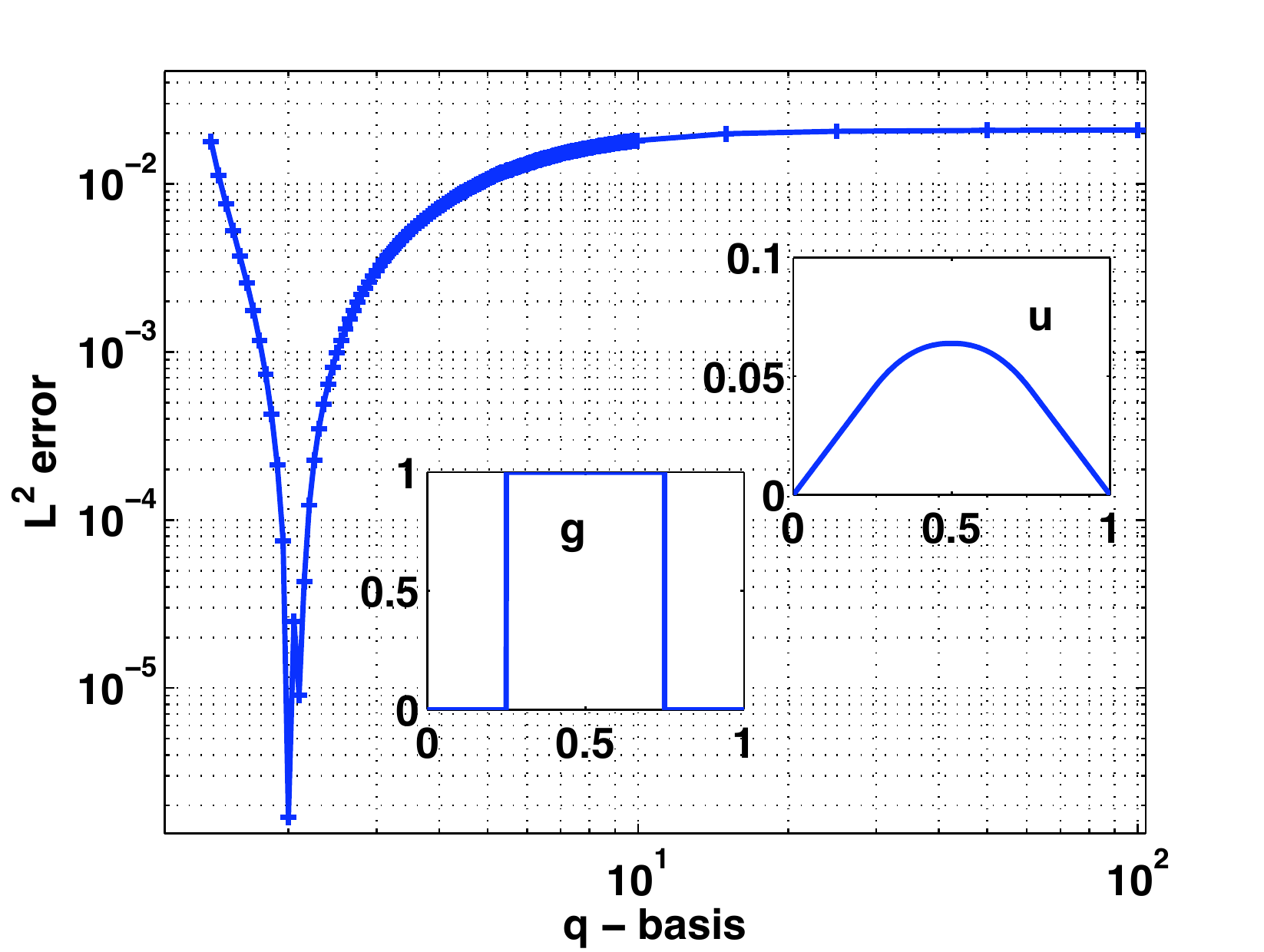}
  \includegraphics*[width=0.48\textwidth]{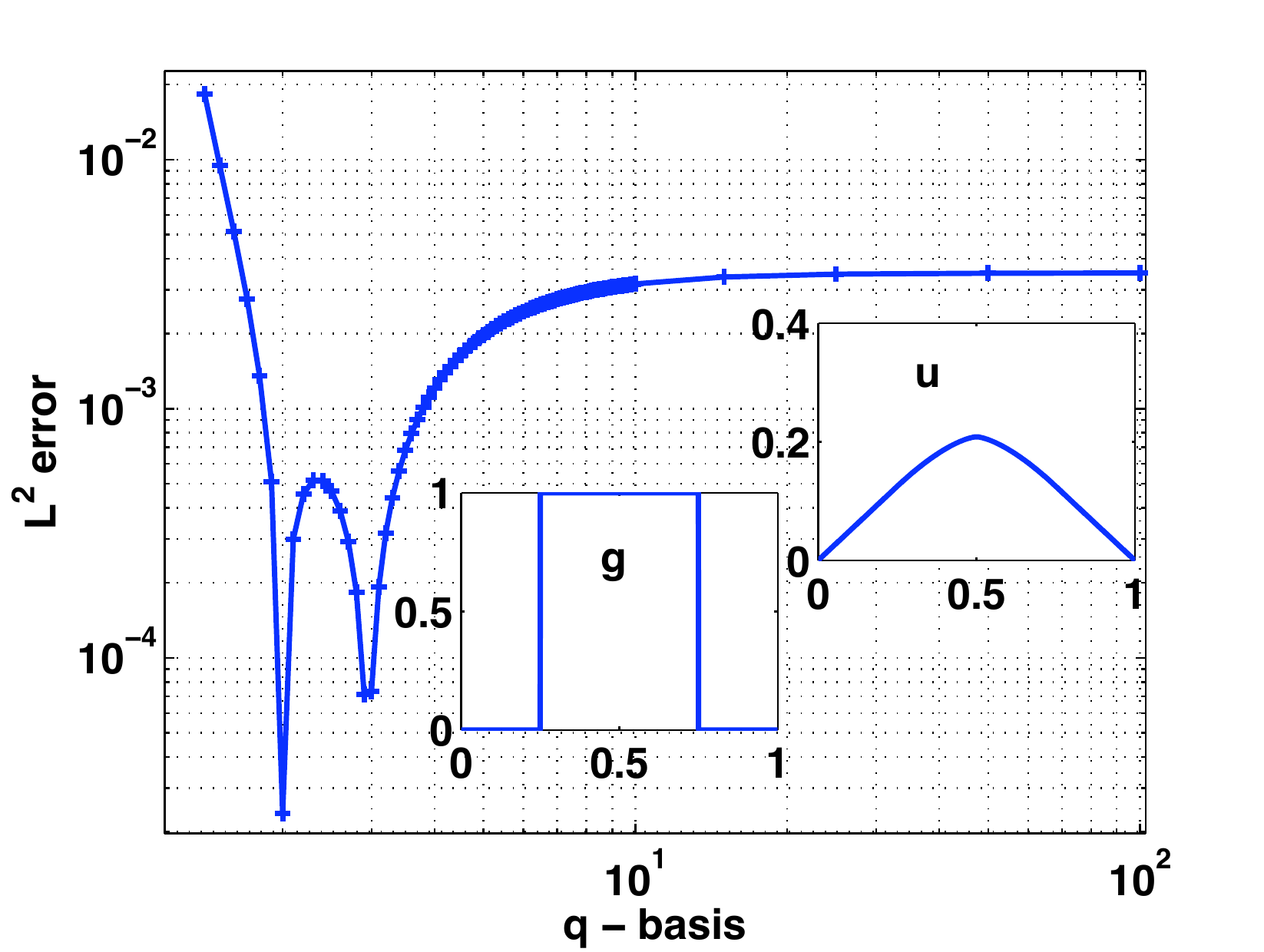}
  (c) \hspace{0.\textwidth} (d) \\
  \includegraphics*[width=0.48\textwidth]{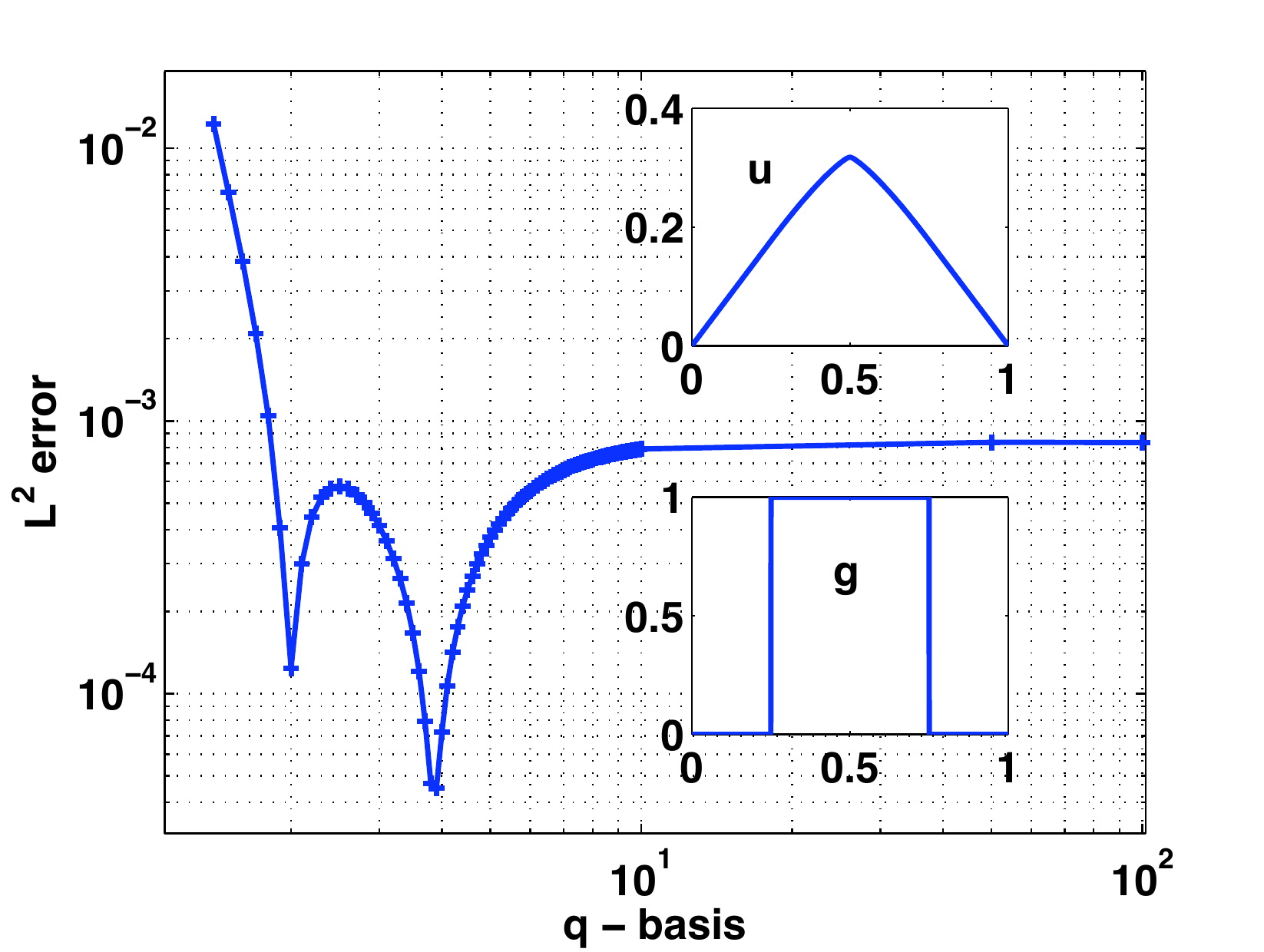}
  \includegraphics*[width=0.48\textwidth]{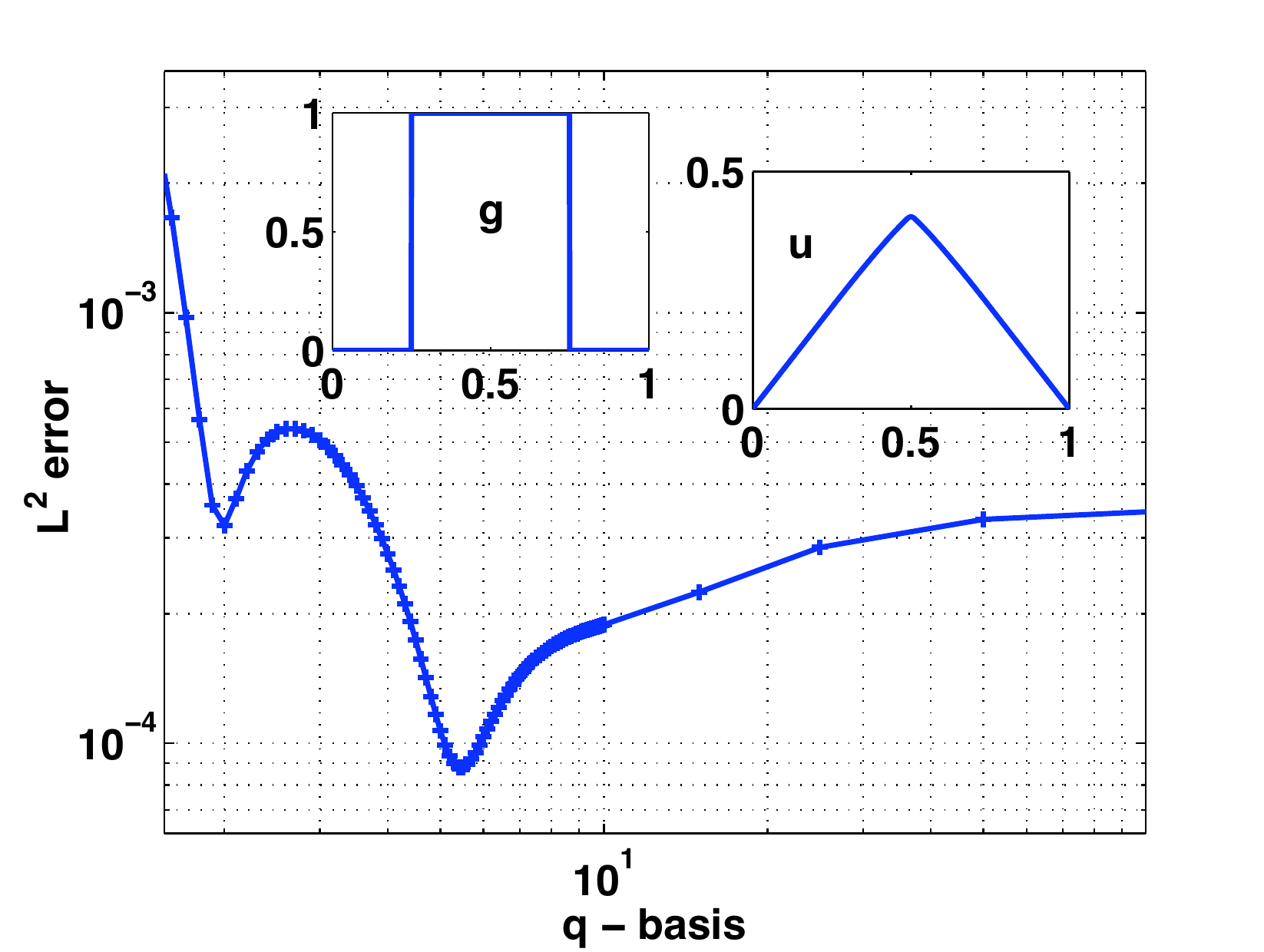}
  \caption{We now examine a piecewise constant $f$ for different $p$ ($N=40$). 
    In (a) $p=1.8$ and $\qopt=2$ (b) $p=3$ and  $\qopt \approx 2$ with 
    another minimum at $q=2.95$,
    in (c) $p=5$ and $\qopt \approx 3.9$ and in (d) $p=10$ and
    $\qopt\approx 5.75$.}
  \label{fig:plap2}
\end{center}
\end{figure}

Numerically the right hand side of \eqref{eq:discrete} is approximated
via quadrature rules. The nonlinear system may then be solved for
example by Newton's method. Since the Jacobian is a full matrix in
this case, a banded approximation appears to provide sufficient
accuracy. However, the results presented below were found using the
trust-region dogleg method with a full Jacobian matrix as implemented
in Matlab.  

\begin{figure}[hth]
\begin{center}
  (a) \hspace{0.48\textwidth} (b)  \\
  \includegraphics*[width=0.4\textwidth]{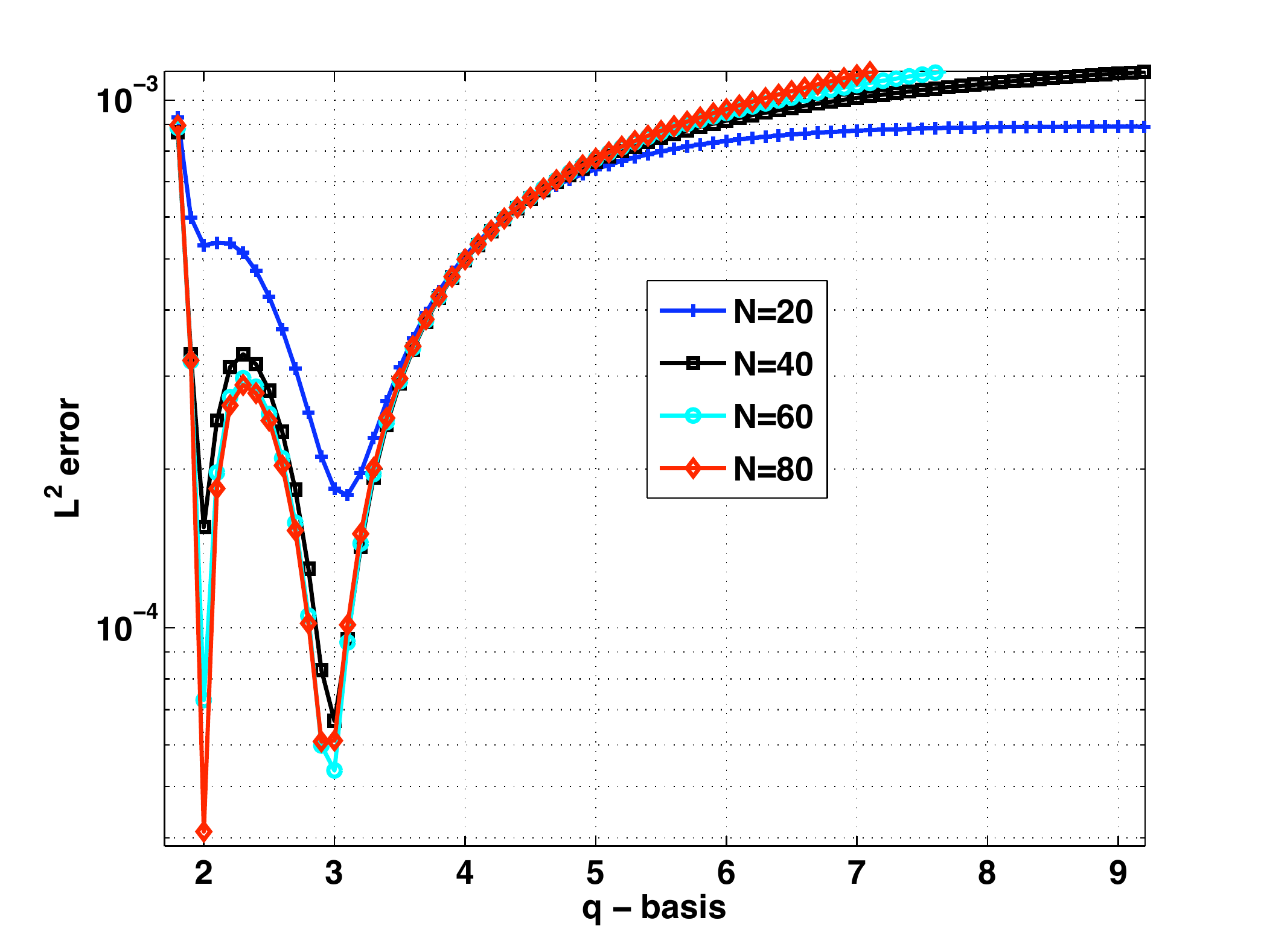}
  \includegraphics*[width=0.4\textwidth]{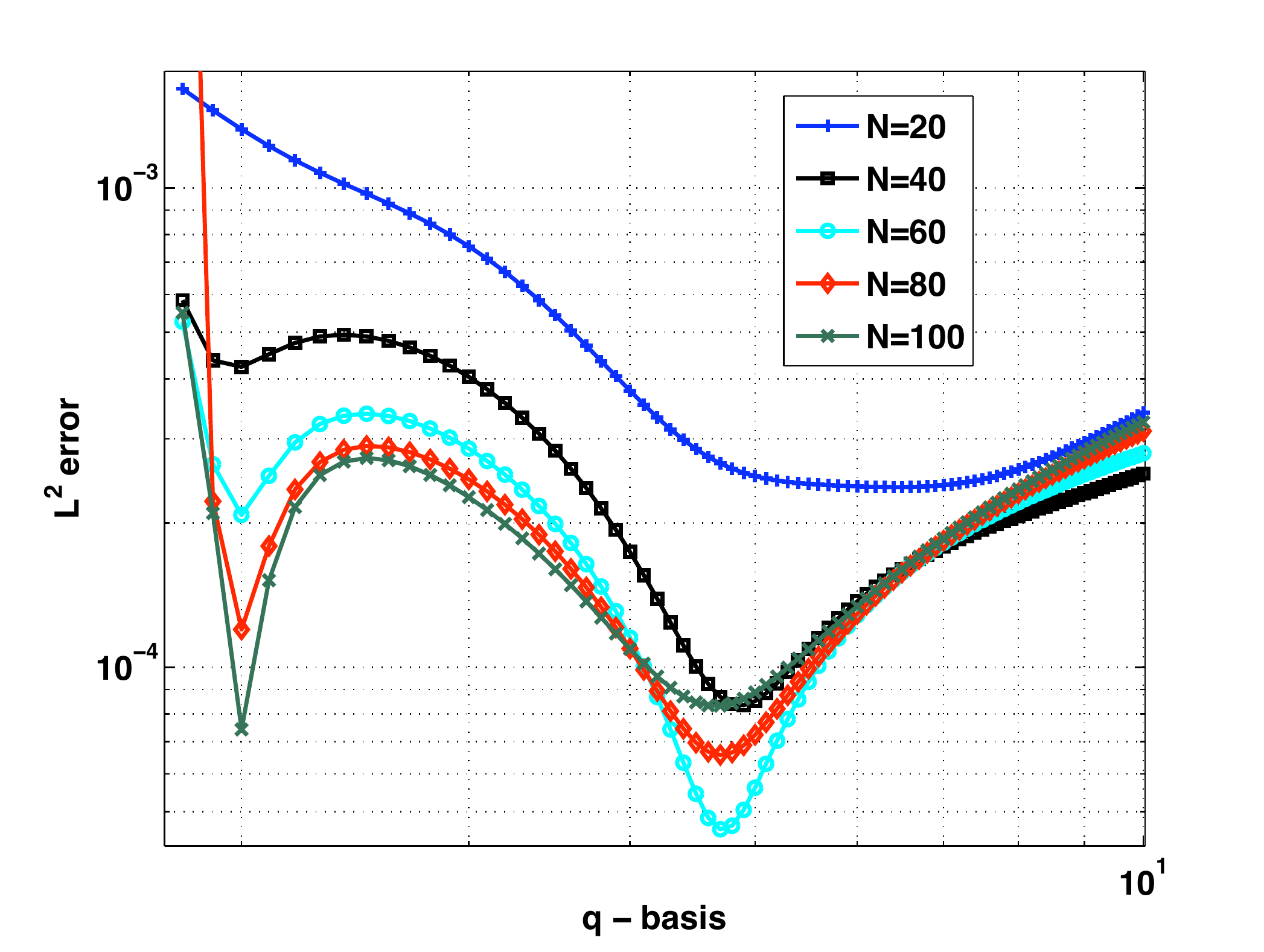}\\
  (c) \hspace{0.48\textwidth} (d)  \\
  \includegraphics*[width=0.4\textwidth]{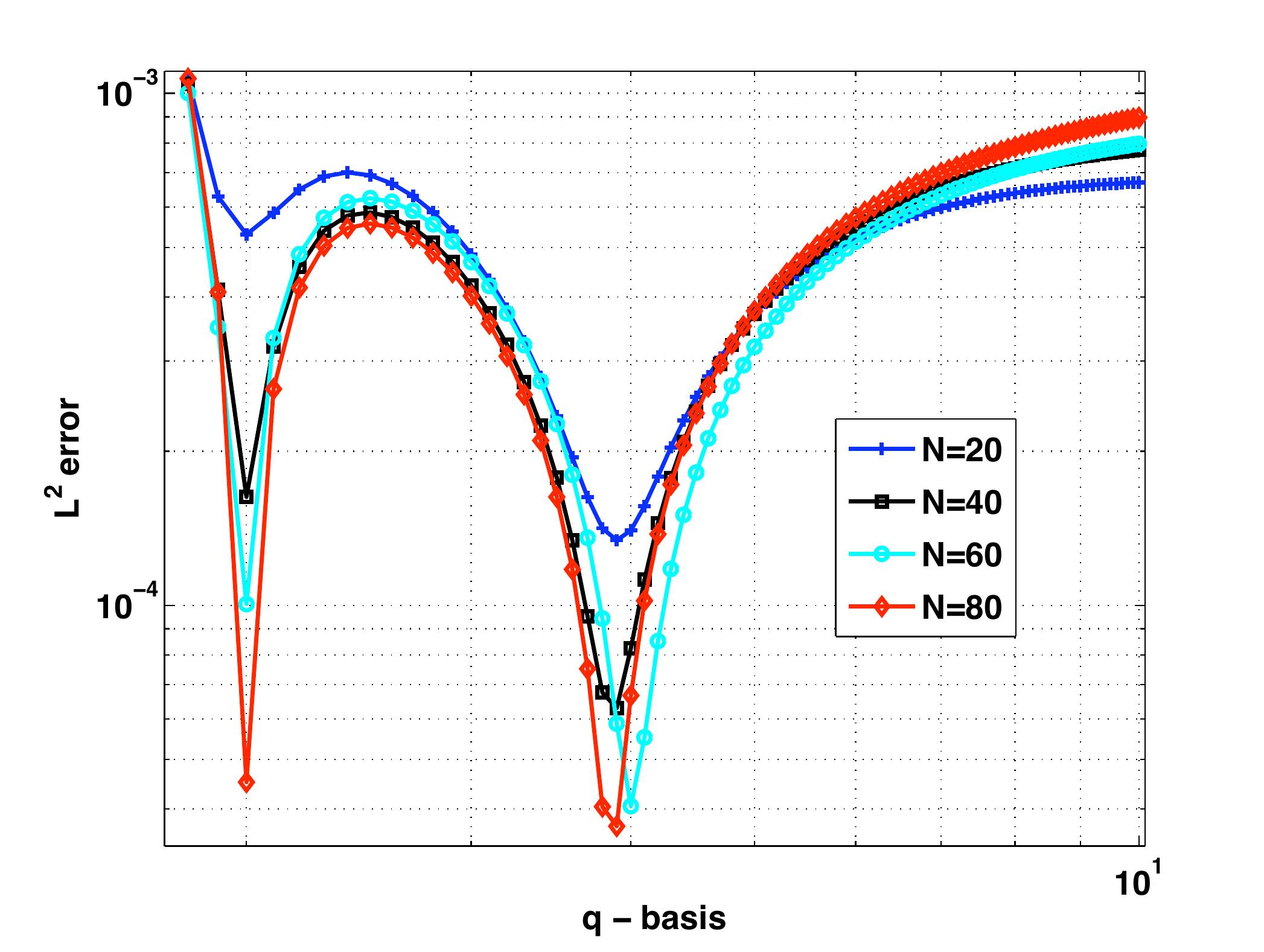}
  \includegraphics*[width=0.4\textwidth]{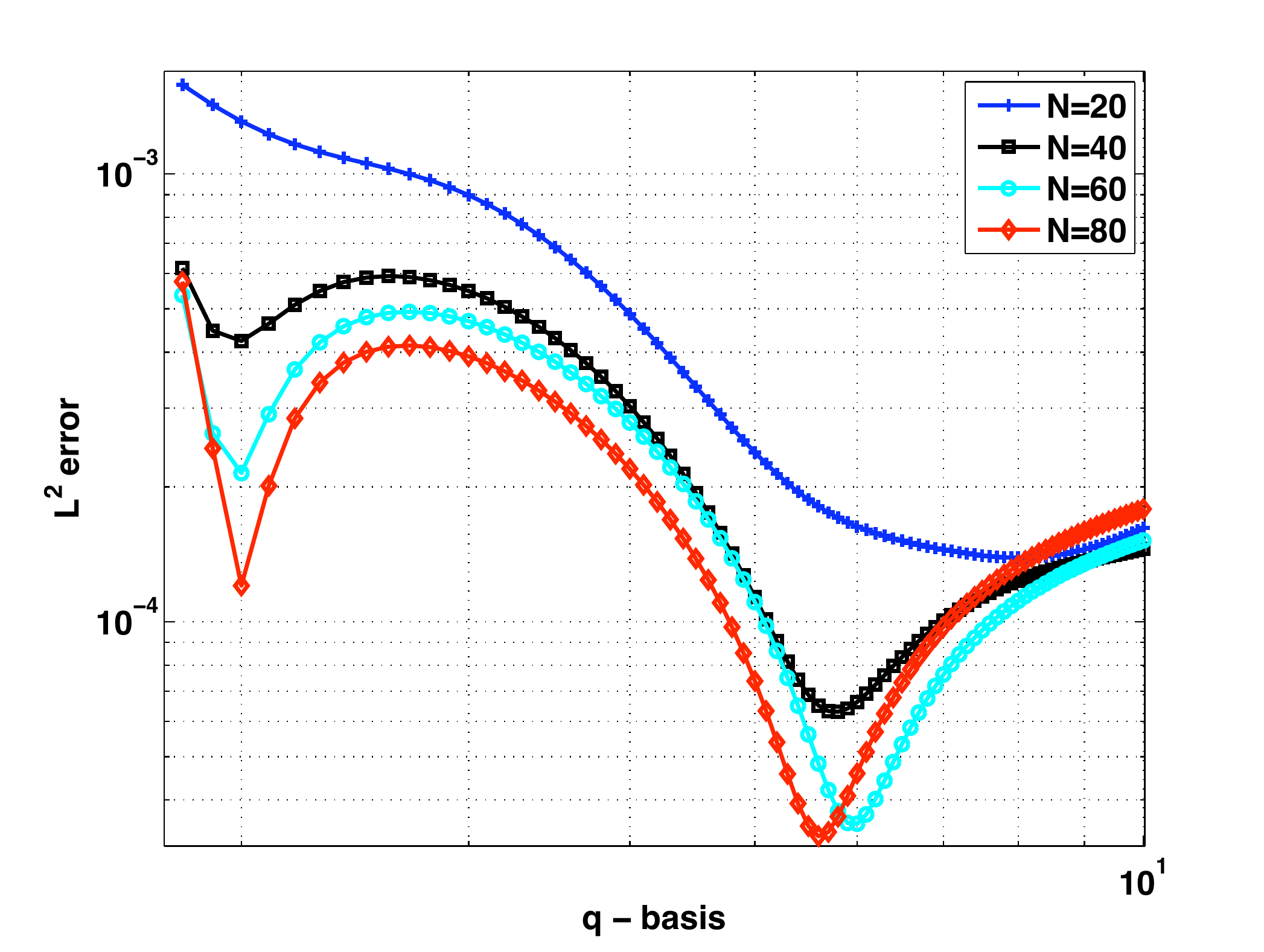}\\
  \caption{Comparison of $N=20,40,60,80$ in (a) for $p=5$ and $g\equiv 1$ 
    (b) for $p=10$ and $g\equiv 1$ (also including $N=100$). In (c) we show for
    $g\equiv g_{\mathrm{b}}$ and $p=5$, and in (d) for  $g\equiv g_{\mathrm{b}}$ and $p=10$.
    We note that for $N$ sufficiently large $\qopt=2$, and a more
    accurate solution may be found for smaller $N$ at $\qopt\neq 2$.
  \label{fig:plap3}}
\end{center}
\end{figure}

In \figref{fig:plap1} we have use this scheme to solve the \pLaplacian
problem with $p=5$ and examine the $L^2$ error taking $N=40$. In \figref{fig:plap1}-(a)
we have fixed a solution $u(x)=\sin(\pi x)$ where we clearly expect
and observe that $q=2$ is the optimal basis. In \figref{fig:plap1}-(b) we have fixed
$u(x)=f_1(x)$ for $p=5$ and so we observe the $q=5$ as the optimal
basis.

\begin{table}[hhh]
\begin{tabular}{l||c|c|c|c}
g              & $p=1.8$& $p=3$ & $p=5$   & $p=10$ \\\hline
$g_{\mathrm{b}}$ $(N=20)$   & 2.05 & 2.95       & 3.9       &  8.0\\
$g_{\mathrm{b}}$ $(N=40)$   & 2.05 & 2.0 (2.95) & 3.9       &  5.75\\
$g_{\mathrm{b}}$ $(N=60)$   & 2.2  & 3.0        & 4.0       &  5.95\\
$g_{\mathrm{b}}$ $(N=80)$   & 2.05 & 2.0        & 3.9       &  5.60\\
\hline
$g=1$ $(N=20)$   & 2.0  & 2.3 & 3.1       &  6.43\\
$g=1$ $(N=40)$   & 2.0  & 2.0 & 3.0       &  4.85 \\
$g=1$ $(N=60)$   & 2.0  & 2.0 & 3.0       &  4.7 \\
$g=1$ $(N=80)$   & 2.0  & 2.0 & 2.0 (2.95)&  4.7 \\ 
\end{tabular}
\caption{\label{tab:qopt} Optimal $\qopt$ for four different \pLaplacian
  problems with either a $g= g_{\mathrm{b}}$ from \eqref{functions} or $g=
  1$. We give estimates for different values of $N$. The change in
  $\qopt$ as $N$ increases occurs as the two minima
  interchange, see figures~\ref{fig:plap2} and \ref{fig:plap3}.}
\vspace*{0.2in}
\end{table}

In \figref{fig:plap2} we take $g=g_{\mathrm{b}}$ from
\eqref{functions} which turns out to be a typical form of forcing for
sandpile problems. We solve the \pLaplacian problem for (a) $p=1.8$,
(b) $p=3$, (c) $p=5$ and $p=10$ with $N=40$ and examine the $L^2$
error in the solution as we vary the $q$ basis.  To estimate this
error we take as exact the solution with 2-sines and $2N$ modes.  We
observe that the optimal basis for representing the solutions is no
longer the standard $q=2$ for $p=3,5,10$.  For problems with moderate
$p$ (for example $p=3$) we see two distinct minima. For larger $p$
problems however, the $q=2$ basis becomes less competitive.

In \tabref{tab:qopt} we give estimates of $\qopt$ for
$g=g_{\mathrm{b}}$ and $g= 1$ for $N=20,40,60$ and $N=80$ modes.  The
changes in $\qopt$ with $N$ are explained by the interchange of the
two minima in \figref{fig:plap2}.  For $N=100$ this interchange occurs
for $p=5$ and $p=10$ as illustrated in \figref{fig:plap3}-(b).

\section{The time dependent \pLaplacian}
\label{5}

In this final section we consider the evolution equation
\eqref{evo_equ} both in the deterministic ($\nu=0$)
and the stochastically forced regime ($\nu \neq 0$).
Slow-fast diffusion is often taken in the large $p$ limit, as a model
for sandpile growth. For further details including arguments about
the validity of the modeling see for example
\cite{MR2240806,MR1419017,Igbida,MR1451539,MR2481827,MR2449104}. Our
purpose here is to consider the time dependent problem for
different values of $p$ and examine the choice of basis in the
formation of the sandpile.

We discretize the weak form of the equation and truncate to solve the
time-dependent version of \eqref{eq:discrete} given by \be
\frac{dc_j}{dt} = \sum_{k=1}^N c_k \int_0^1 \left| \sum_{\ell=1}^N
  c_\ell \phi_\ell'\right|^{p-2}\phi_k'\phi_j' dx - \langle g,\phi_j
\rangle + \nu \langle \frac{dW}{dt} ,\phi_j \rangle \qquad j=1:N.
\label{eq:semi-discrete}
\ee 
This expression is then discretized in time to get a nonlinear
system of equations to solve for $c_j^n$ at each step. Below we
consider an Euler's method which reduces to
\begin{align*}
  \frac{c^{n+1}_j-c^n_j}{\Dt} = \sum_{k=1}^N c_k^{n+1} & \int_0^1
  \left| \sum_{\ell=1}^N c_\ell^{n+1}
    \phi_\ell'\right|^{p-2}\phi_k'\phi_j' dx \\ & - \langle \phi_j^*,g
  \rangle + \nu \langle \phi_j^*,\sum_{m=1}^M
  \alpha_j^{1/2}e_j(x)\Delta\beta_m \rangle
\label{eq:fully-discrete}
\end{align*}
where $\Delta\beta_m$ are independent identically distributed random
variables with mean zero and variance $\Dt$ (recall \eqref{wiener}), and we have assumed the
eigenfunctions of $Q$ are now given by $e_n$. For the case of
stochastic forcing we take $M>>N$.

\begin{figure}[hth]
\begin{center}
  (a) \hspace{0.48\textwidth} (b)  \\
  \includegraphics*[width=0.24\textwidth]{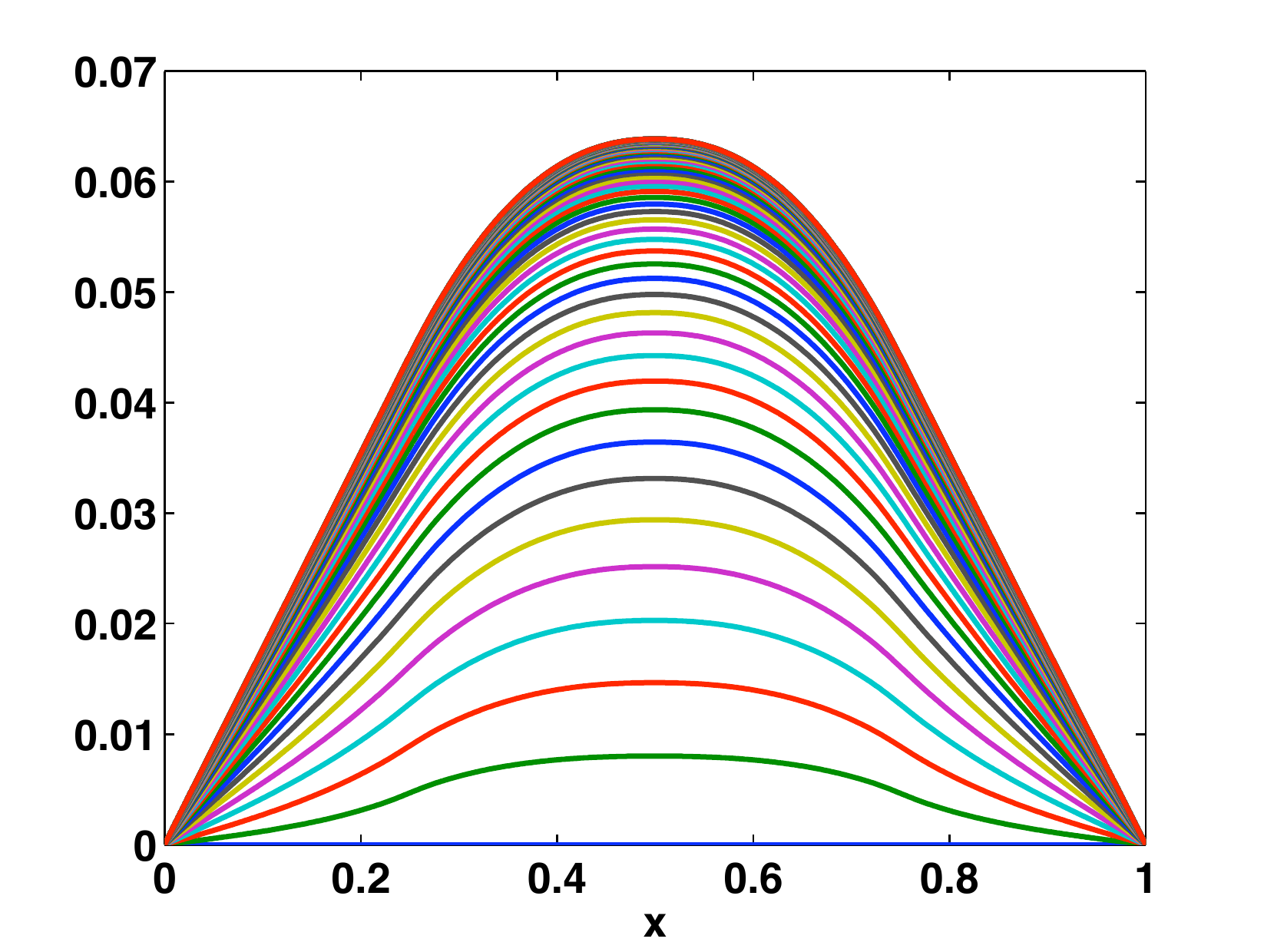}
  \includegraphics*[width=0.24\textwidth]{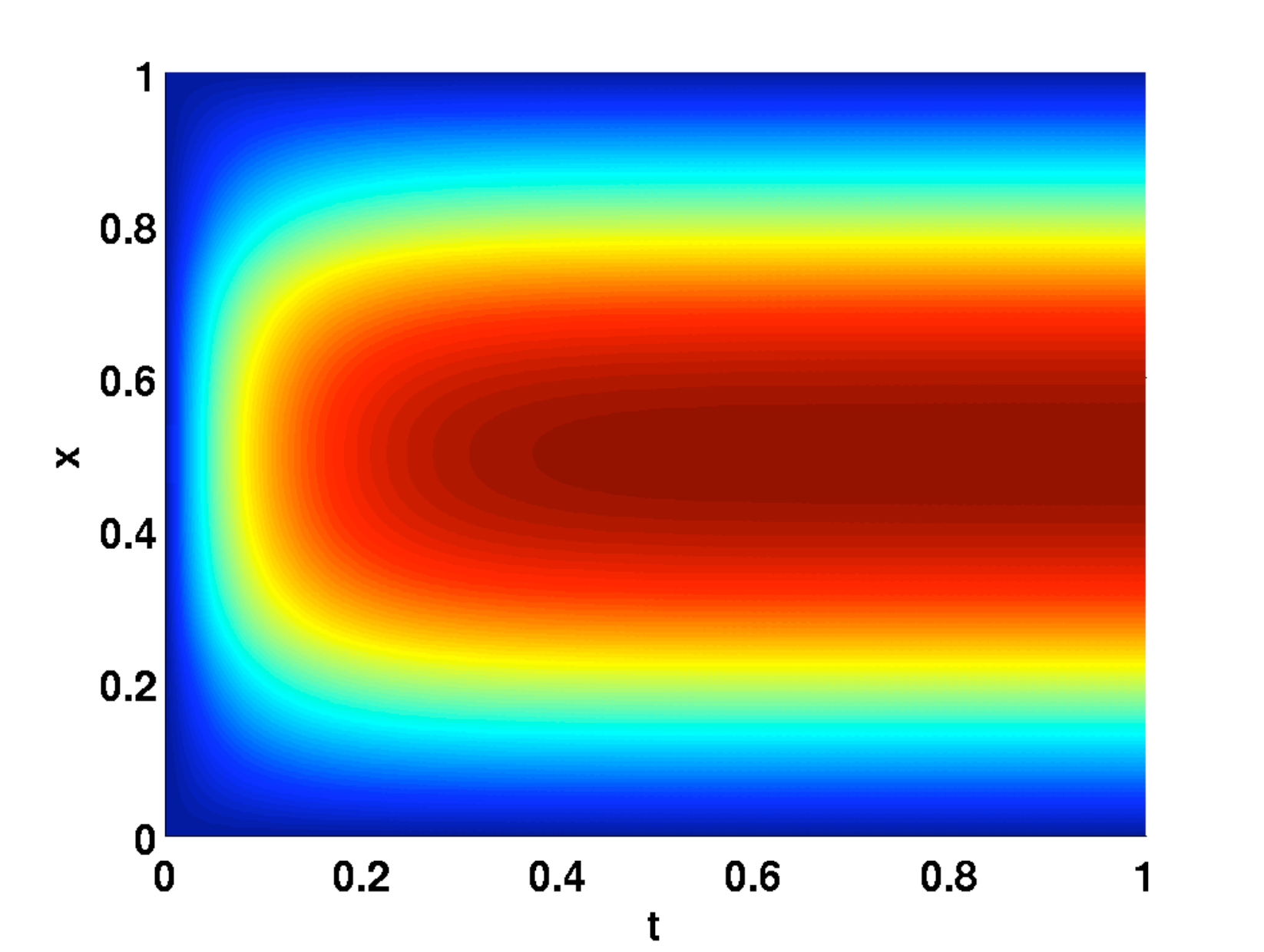}
  \includegraphics*[width=0.24\textwidth]{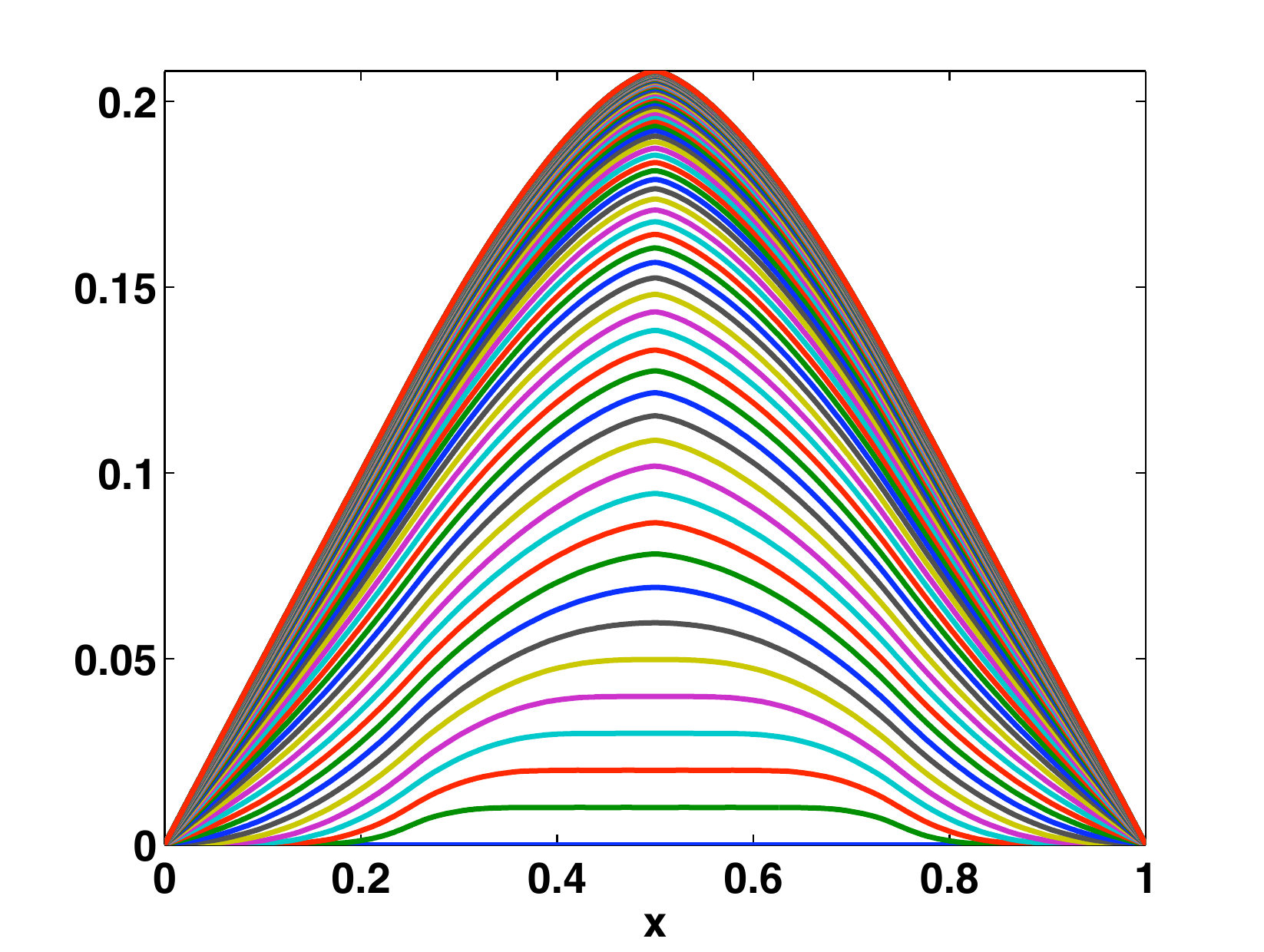}
  \includegraphics*[width=0.24\textwidth]{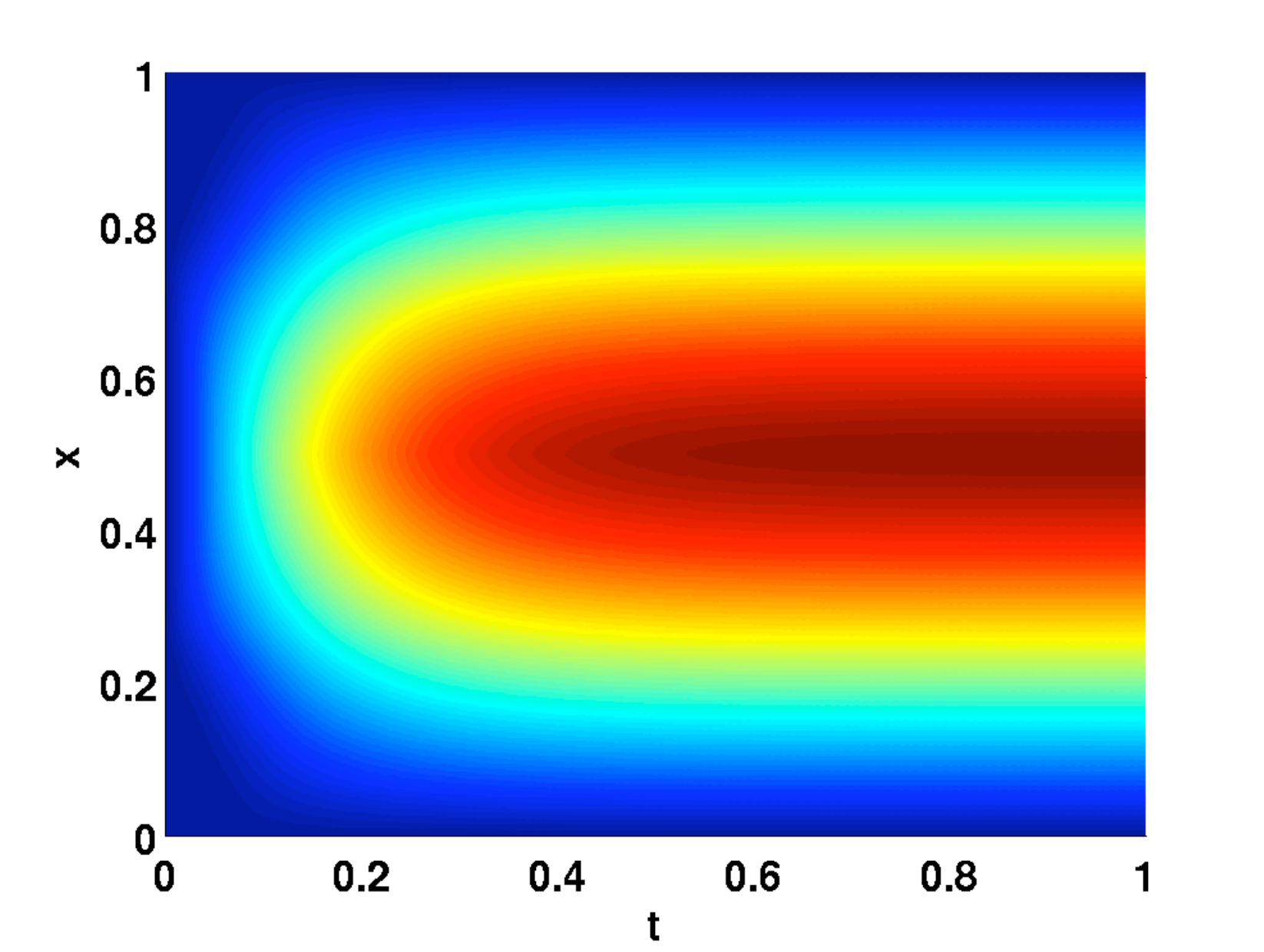}
  (c) \hspace{0.48\textwidth} (d)  \\
  \includegraphics*[width=0.24\textwidth]{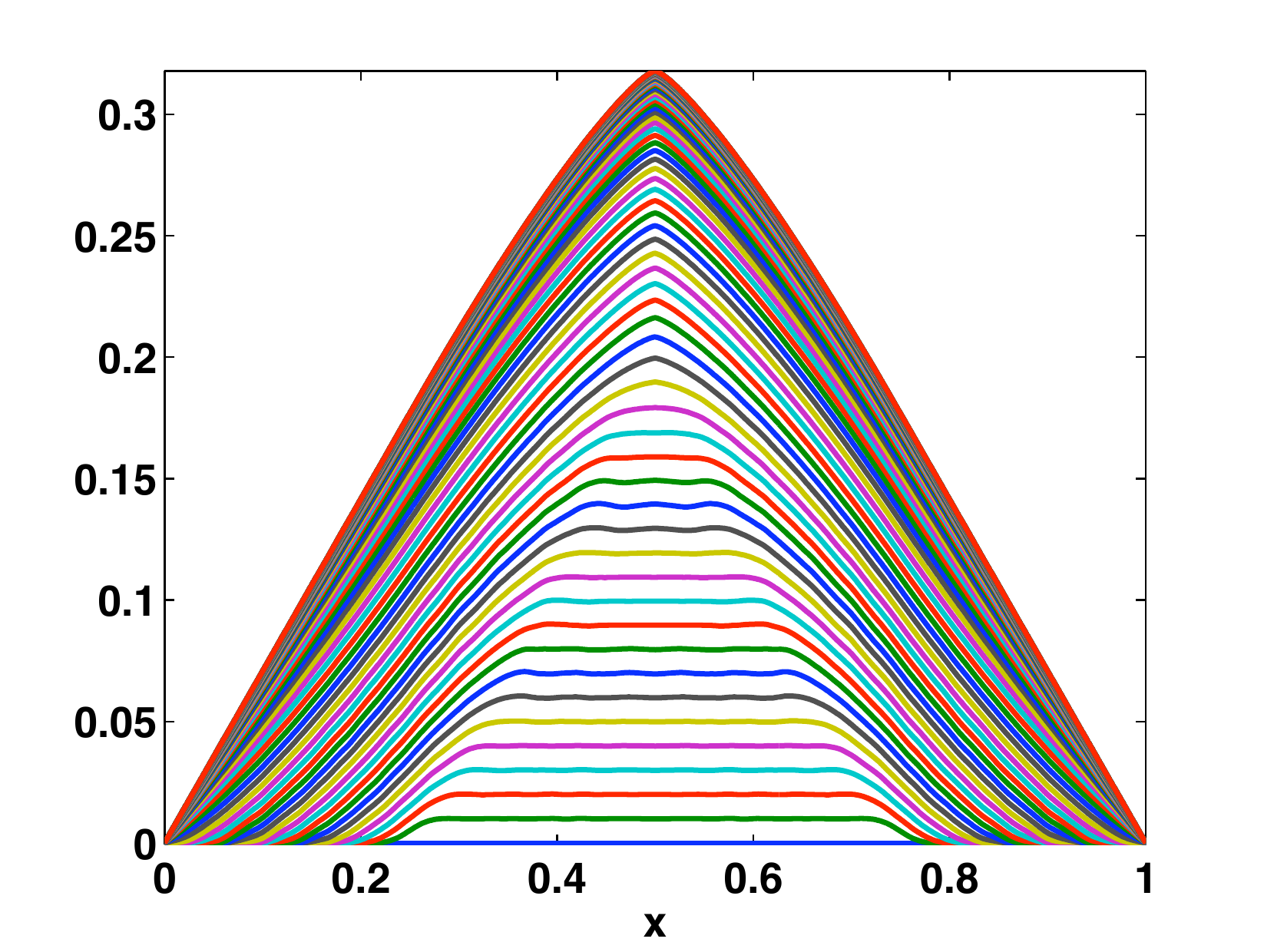}
  \includegraphics*[width=0.24\textwidth]{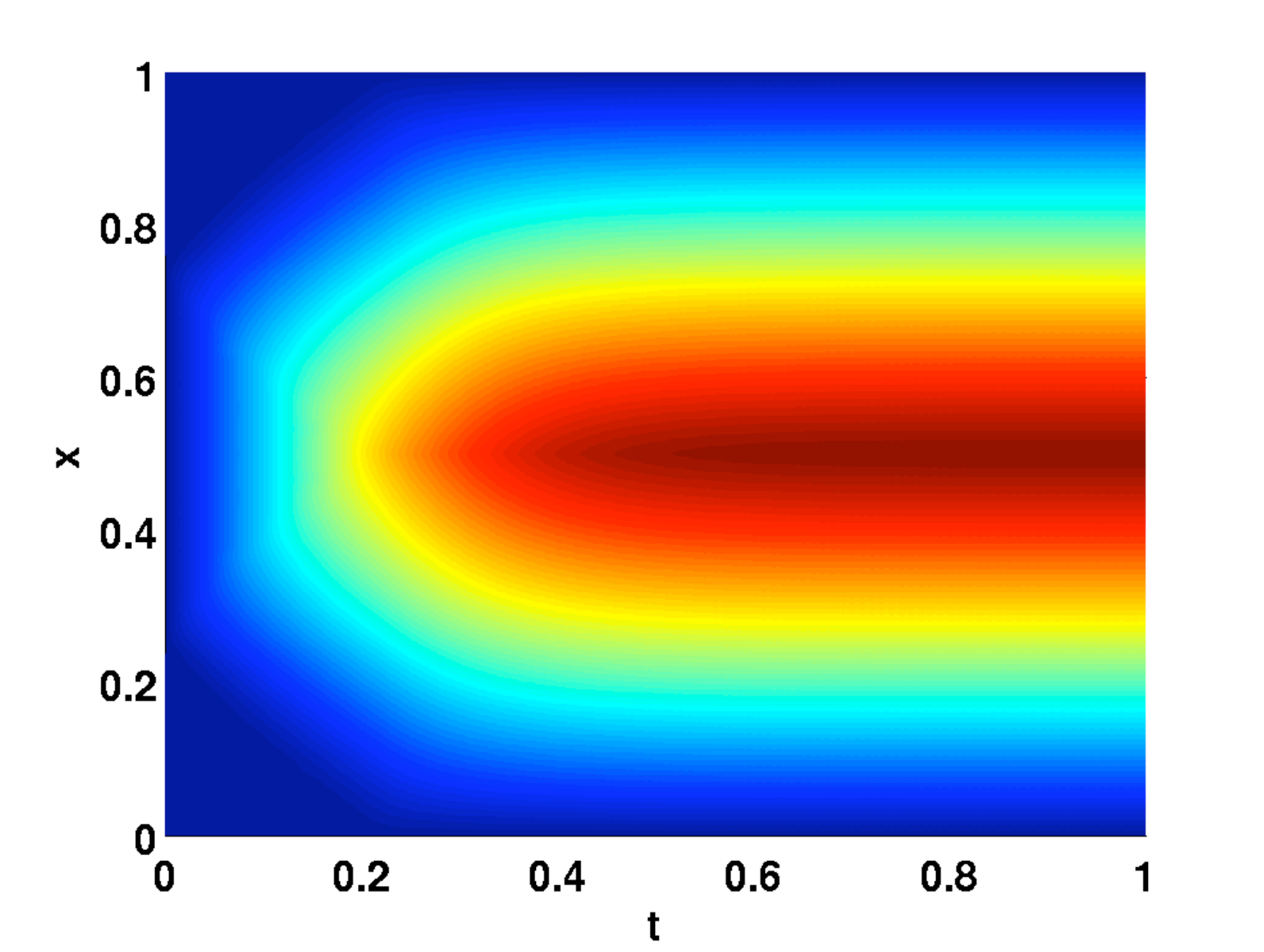}
  \includegraphics*[width=0.24\textwidth]{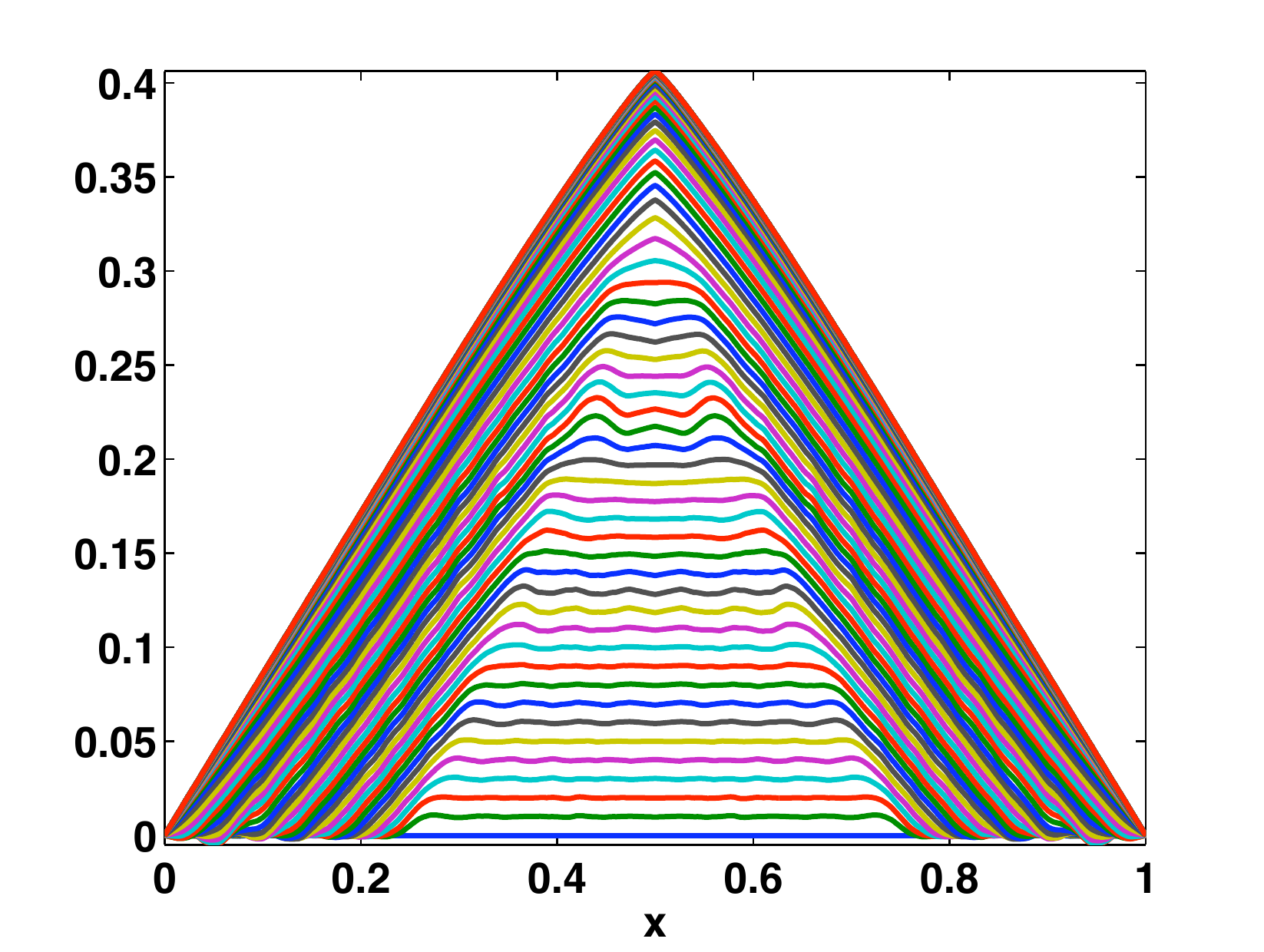}
  \includegraphics*[width=0.24\textwidth]{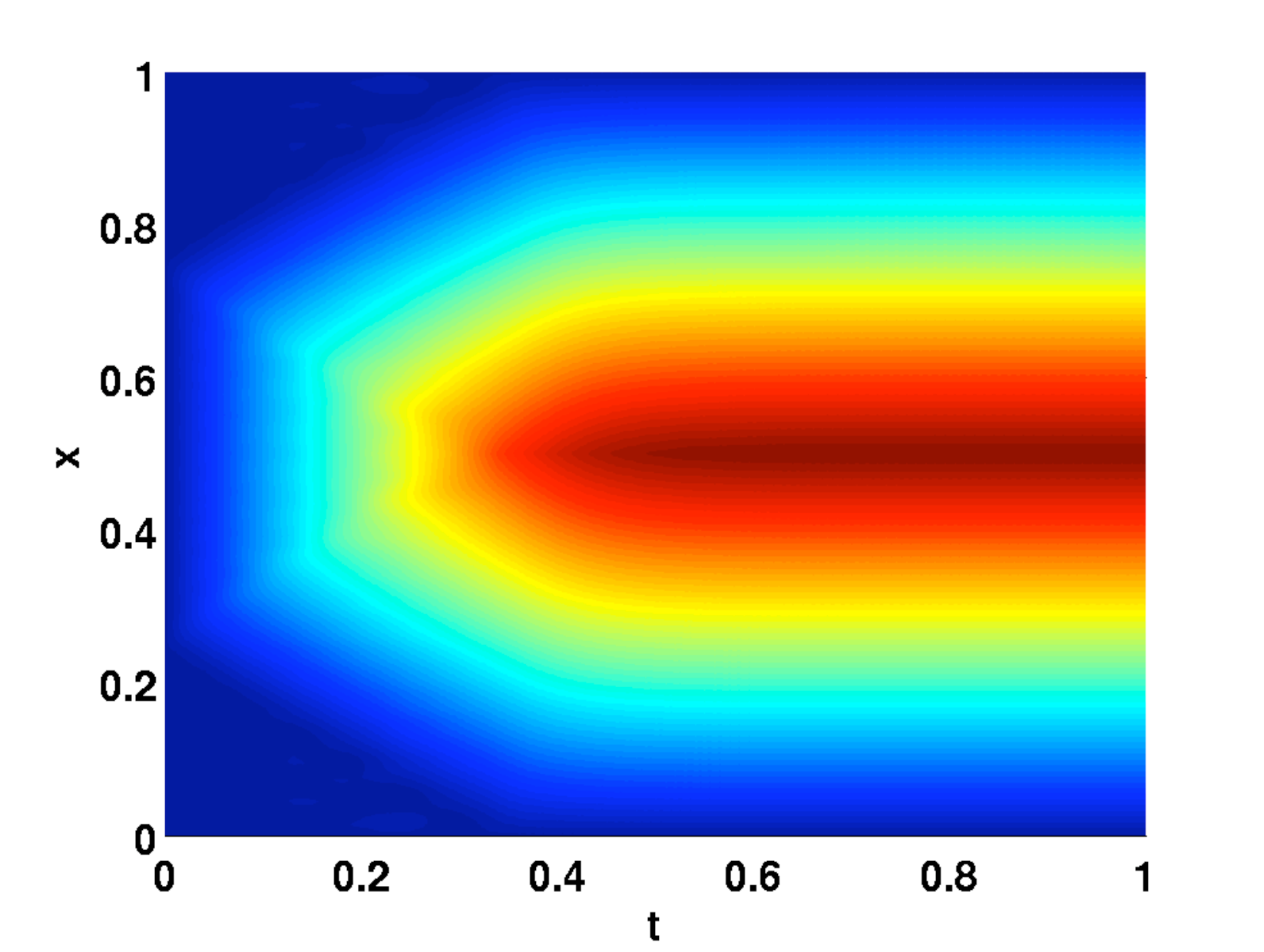}
  \caption{Solution of \eqref{evo_equ} with $g=g_{\mathrm{b}}$ for (a) $p=1.8$ 
    using $q=2$, (b) $p=3$ using $q=2.95$, (c)
    $p=5$ using $q=3.9$ and (d)  $p=10$ using $q=5.75$. As time
    evolves the solutions approach the steady state which, for 
    increasing $p$, turns out to be close to a hat function.}
  \label{fig:plaptime2}
\end{center}
\end{figure}

In \figref{fig:plaptime2} we consider $\nu=0$ with $g=g_{\mathrm{b}}$
and we plot the time evolution for (a) $p=1.8$ ($\qopt=2$),
(b) $p=3$ ($\qopt=2.95$), (c) $p=5$ ($\qopt=3.9$) and (d) $p=10$
($\qopt=5.75$). In each case the evolution found numerically
quickly converges towards the steady state. 

\begin{figure}[hth]
  \begin{center}
    (a) \hspace{0.4\textwidth} (b)  \\
    \includegraphics*[width=0.4\textwidth]{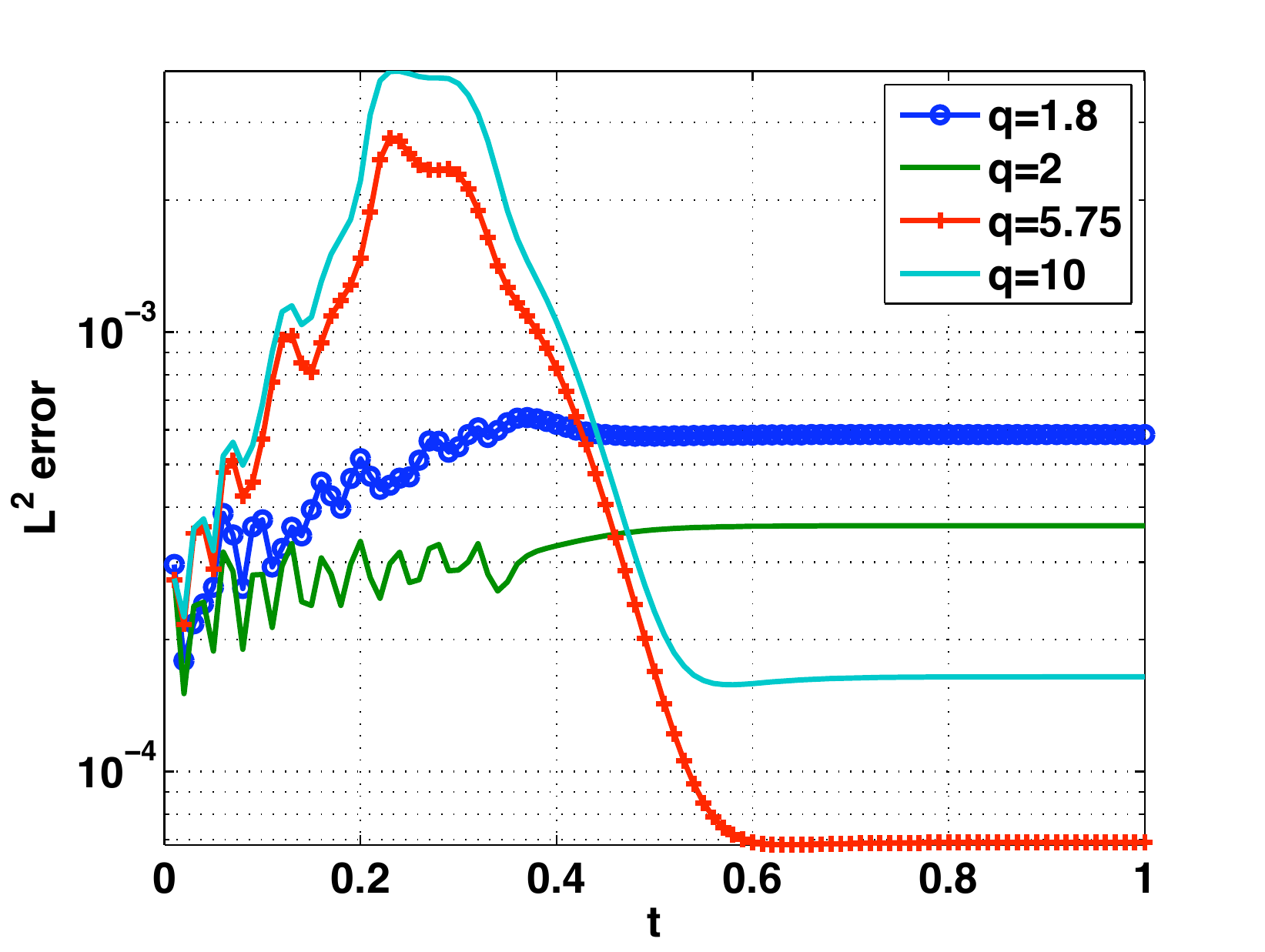}
    \includegraphics*[width=0.4\textwidth]{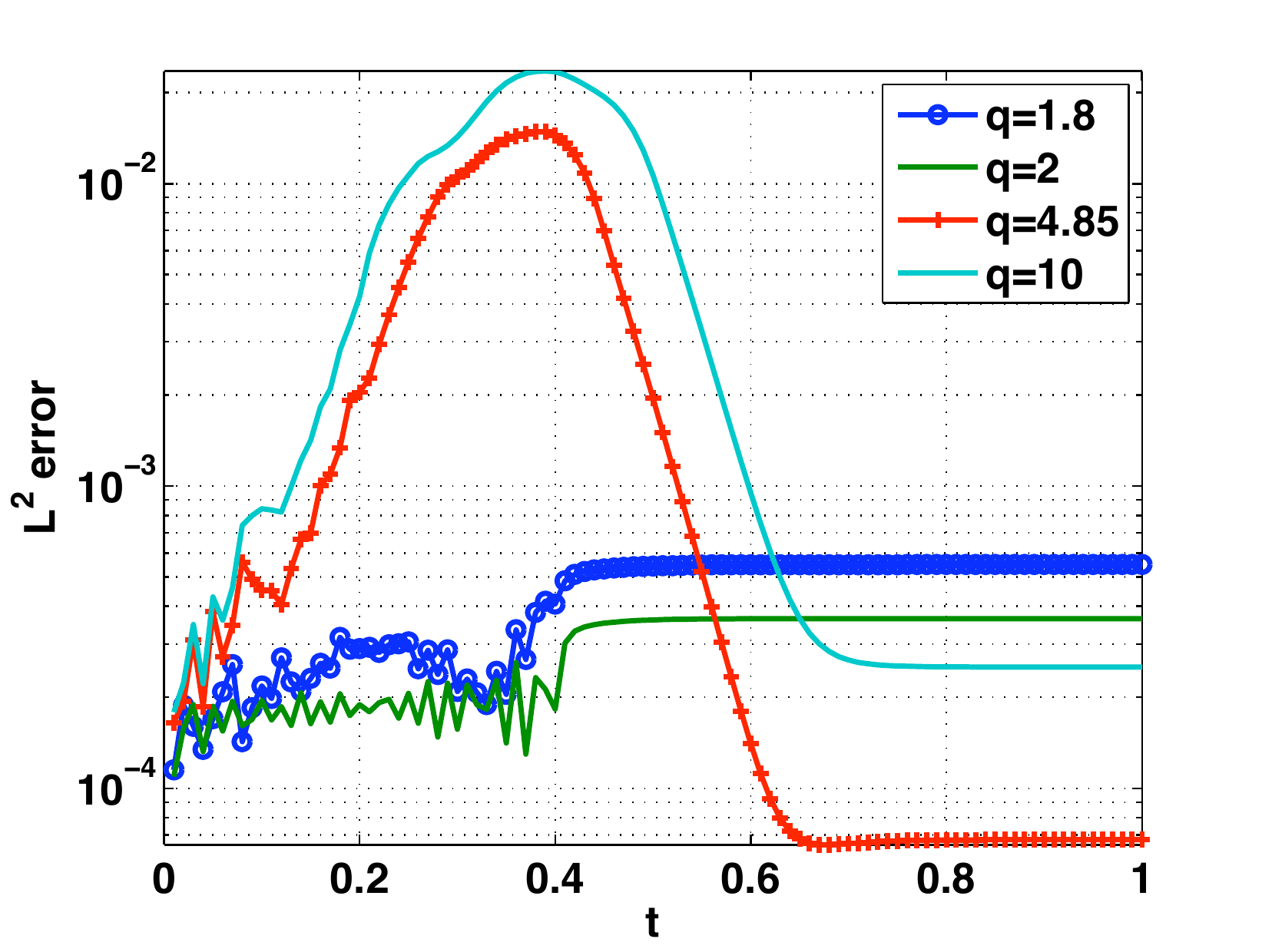}
    \caption{Residual in time for the numerical solution of
      \eqref{evo_equ}. We choose $\nu=0$, $N=40$ and take as a true
      solution one computed with $N=100$ and $q=2$. We fix $p=10$ and
      compare basis with 
      $q=1.8$, $q=2$, $q=\qopt$ and $q=10$. Here (a) corresponds to
      $g=g_\mathrm{b}$ ($\qopt=5.75$) and (b) to $g= 1$ 
      ($\qopt=4.85$).  We see that asymptotically in time $\qopt$ has
      the smallest error.  At a few
      points in time during the transient state, the $q=1.8$ basis is
      more accurate than the $q=2$.}
    \label{fig:plaptimeN}
  \end{center}
\end{figure}

\Figref{fig:plaptime2} was obtained by choosing in each case the
predicted $q=\qopt$ basis for the steady state. Intuitively this
choice of basis should be near optimal for large $t$. On the other
hand however, there is no reason to presume that it is so for small
values of $t$.  In \figref{fig:plaptimeN} we examine the spatial error
at each step for four different bases with $N=40$. For comparison we
have chosen $p=10$ and $q\in\{1.8, 2, \qopt,10\}$.  As the solution
approaches the steady state, in each case the basis $\qopt$ has the
smallest error as is expected (in fact a gain of almost an order of
magnitude in accuracy compared to $q=2$ is observed). 
It is remarkable however that, in the transient regime, these bases appear to be far
less accurate than other choices of $q$, such as $q=2$ and
$q=1.8$.  

\begin{figure}[hhh]
  \begin{center}
    (a) \hspace{0.32\textwidth} (b)  \hspace{0.32\textwidth} (c) \\
    \includegraphics*[width=0.32\textwidth]{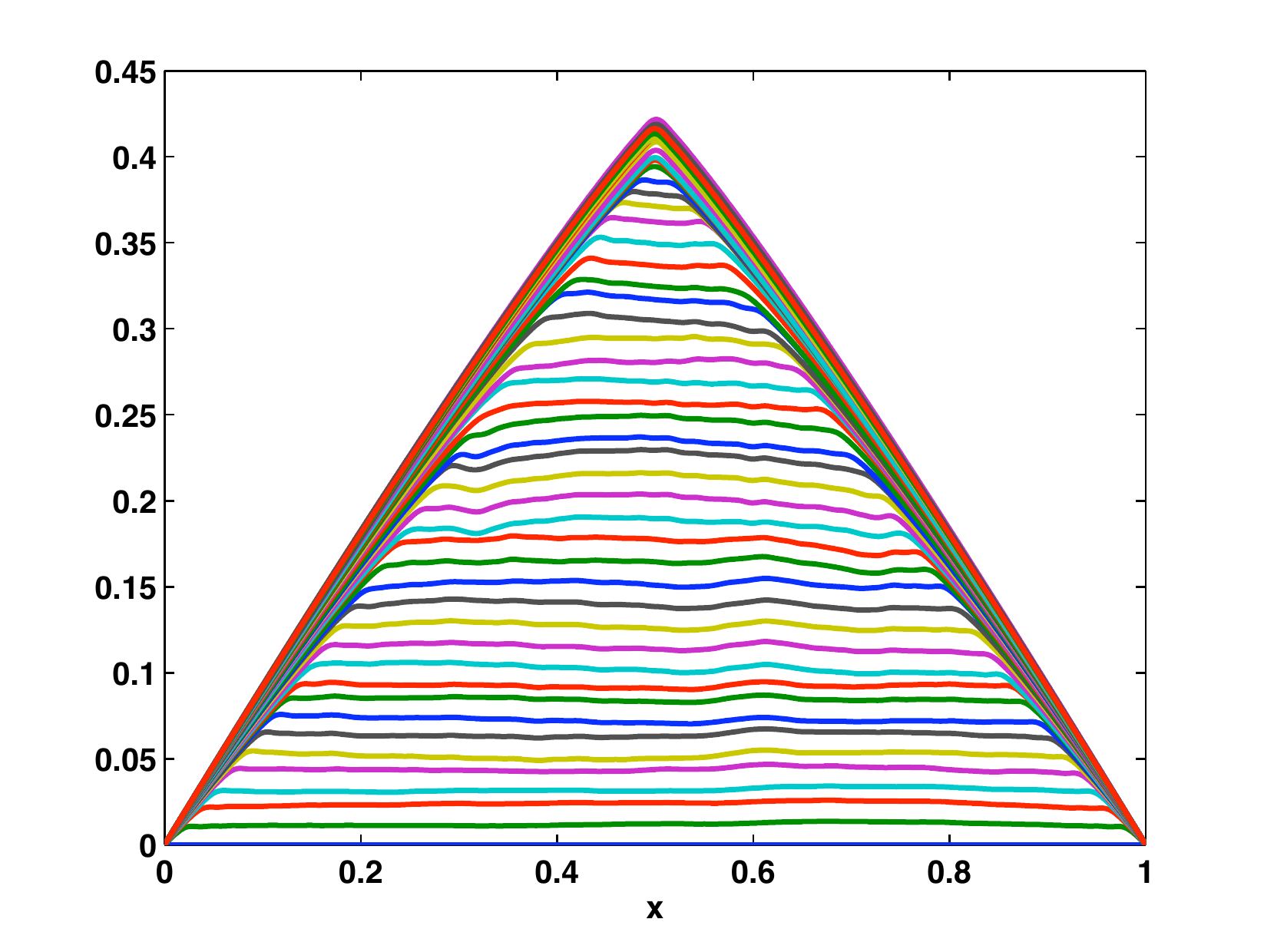}
    \includegraphics*[width=0.32\textwidth]{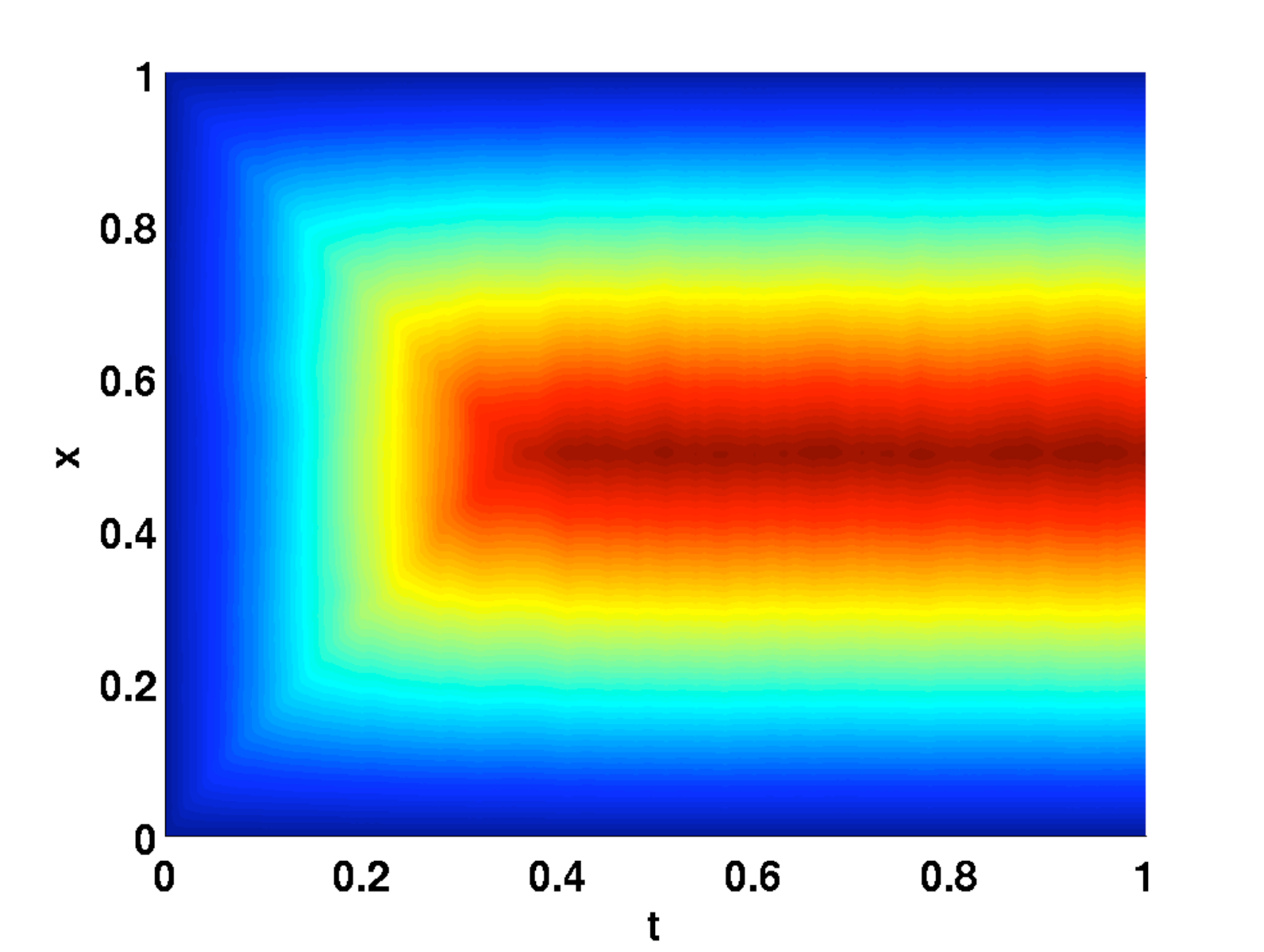}
    \includegraphics*[width=0.32\textwidth]{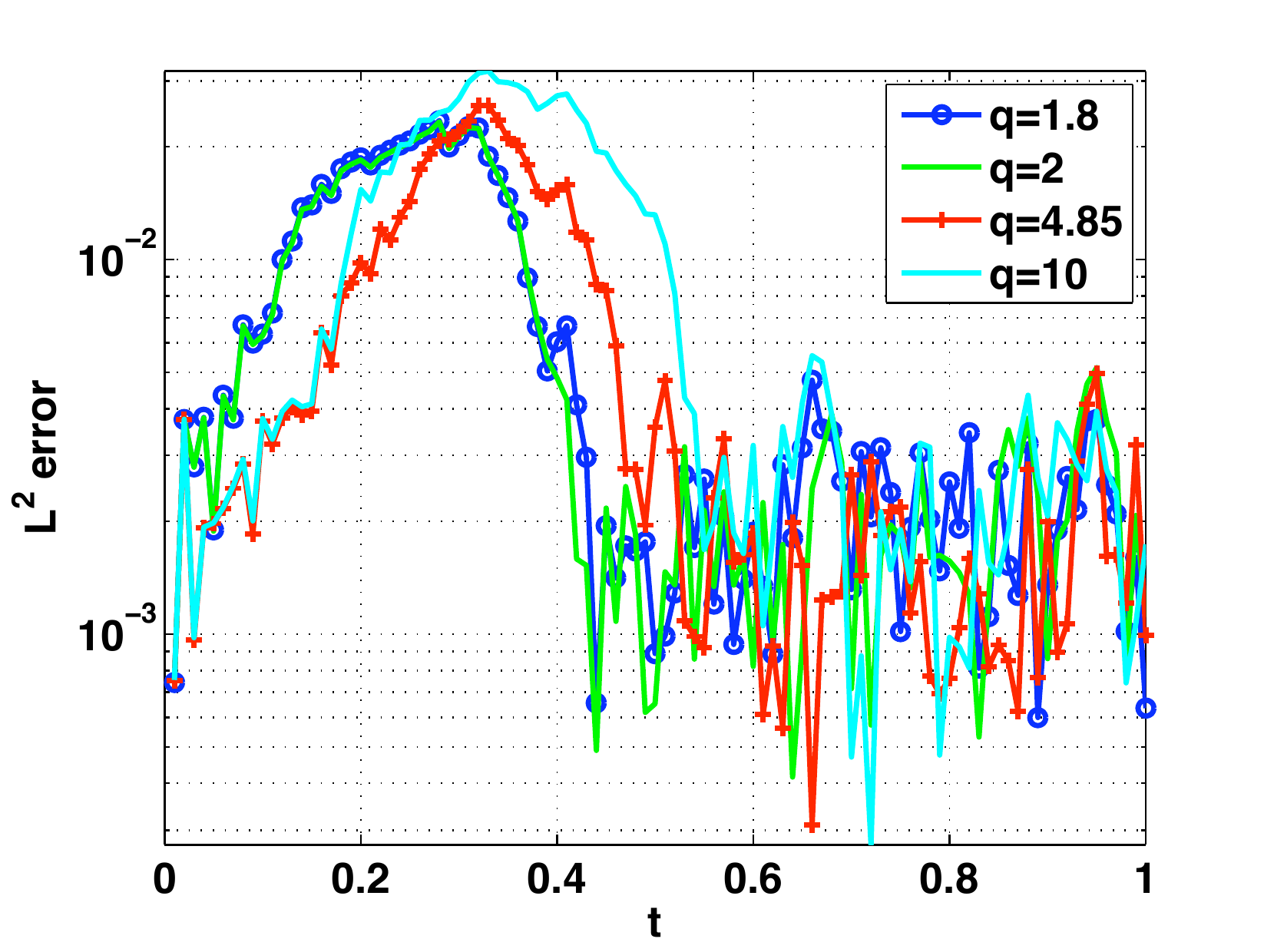}\\
    (d) \hspace{0.32\textwidth} (e)  \hspace{0.32\textwidth} (f) \\
    \includegraphics*[width=0.32\textwidth]{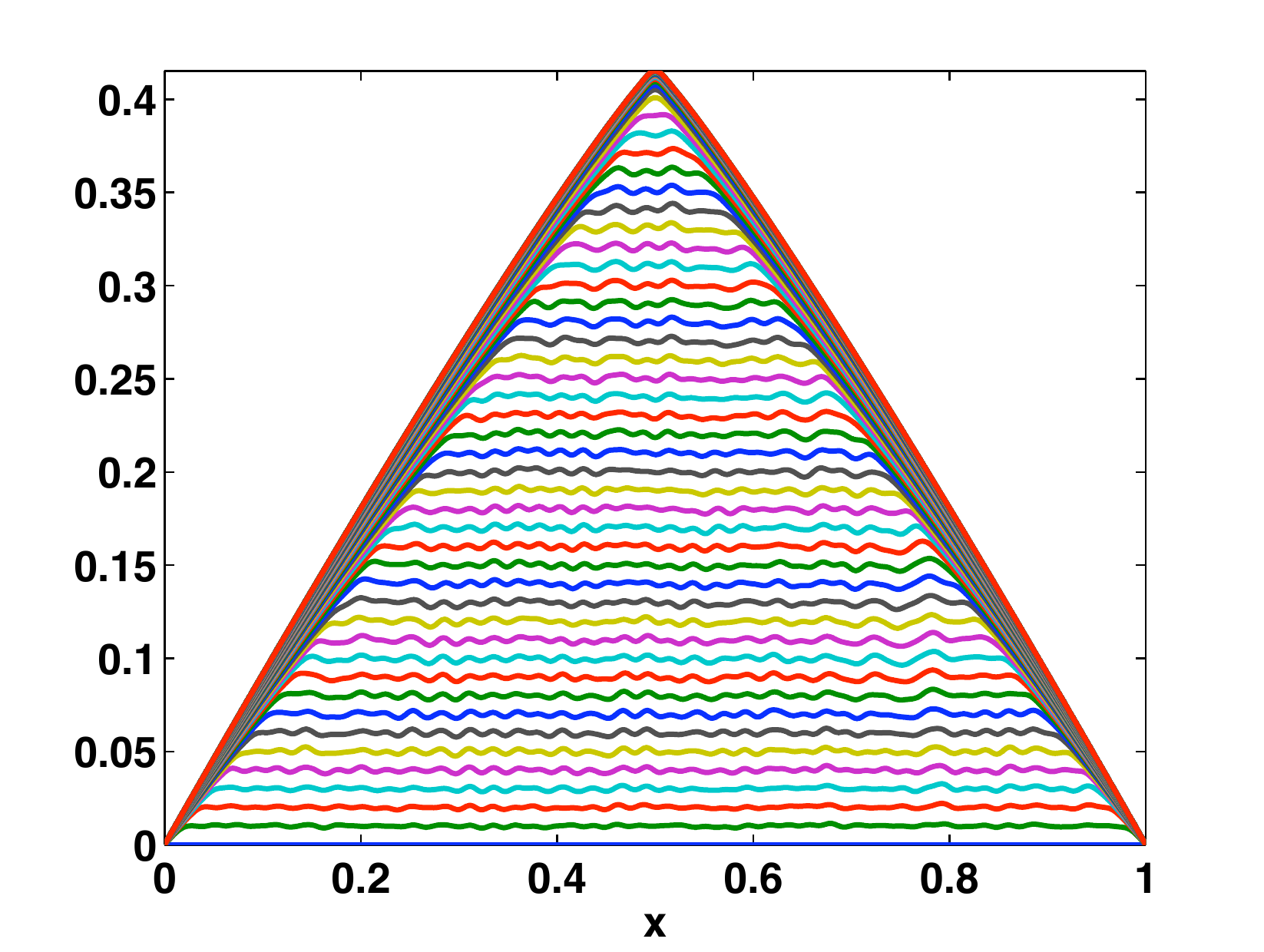}
    \includegraphics*[width=0.32\textwidth]{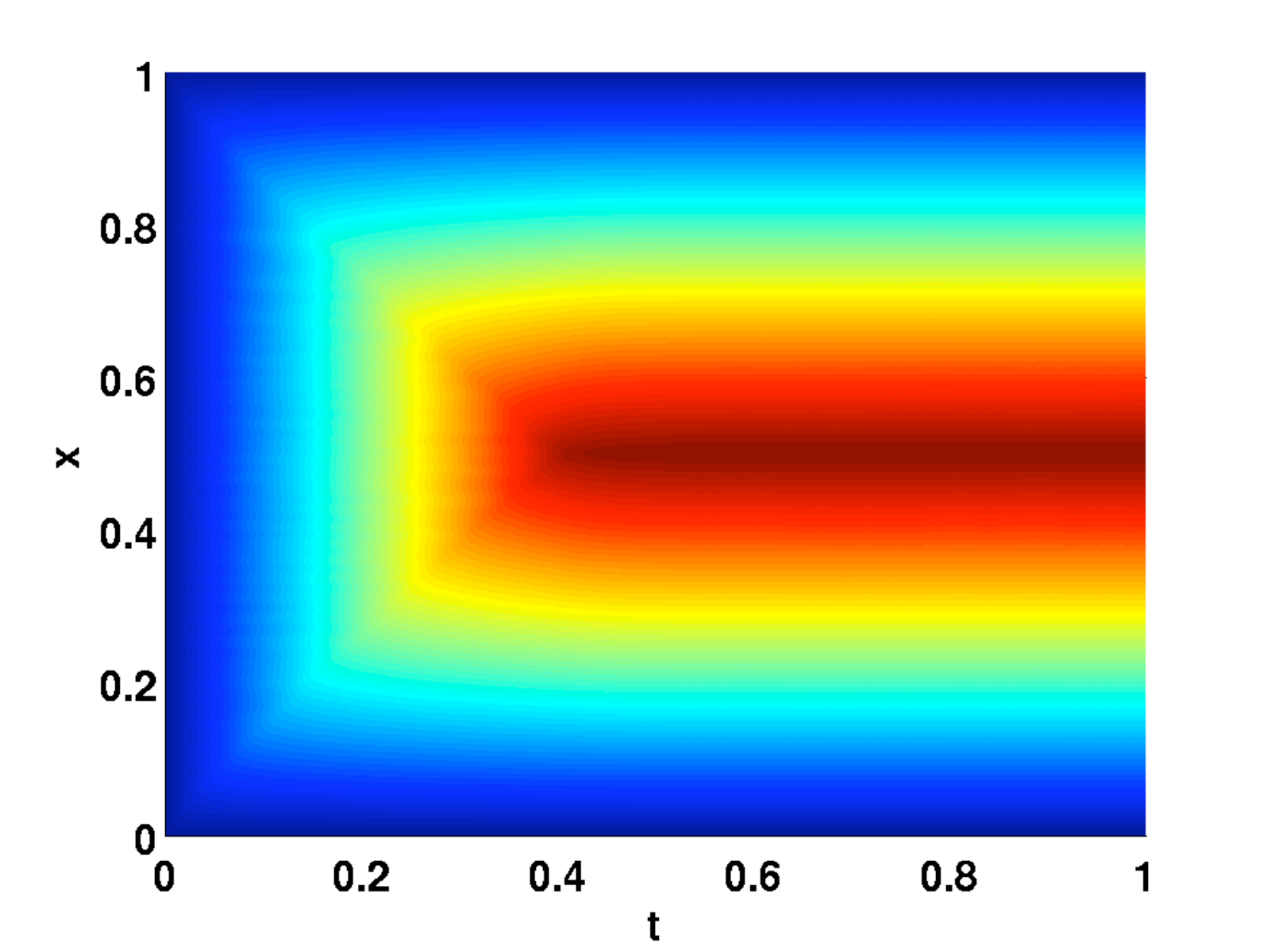}
    \includegraphics*[width=0.32\textwidth]{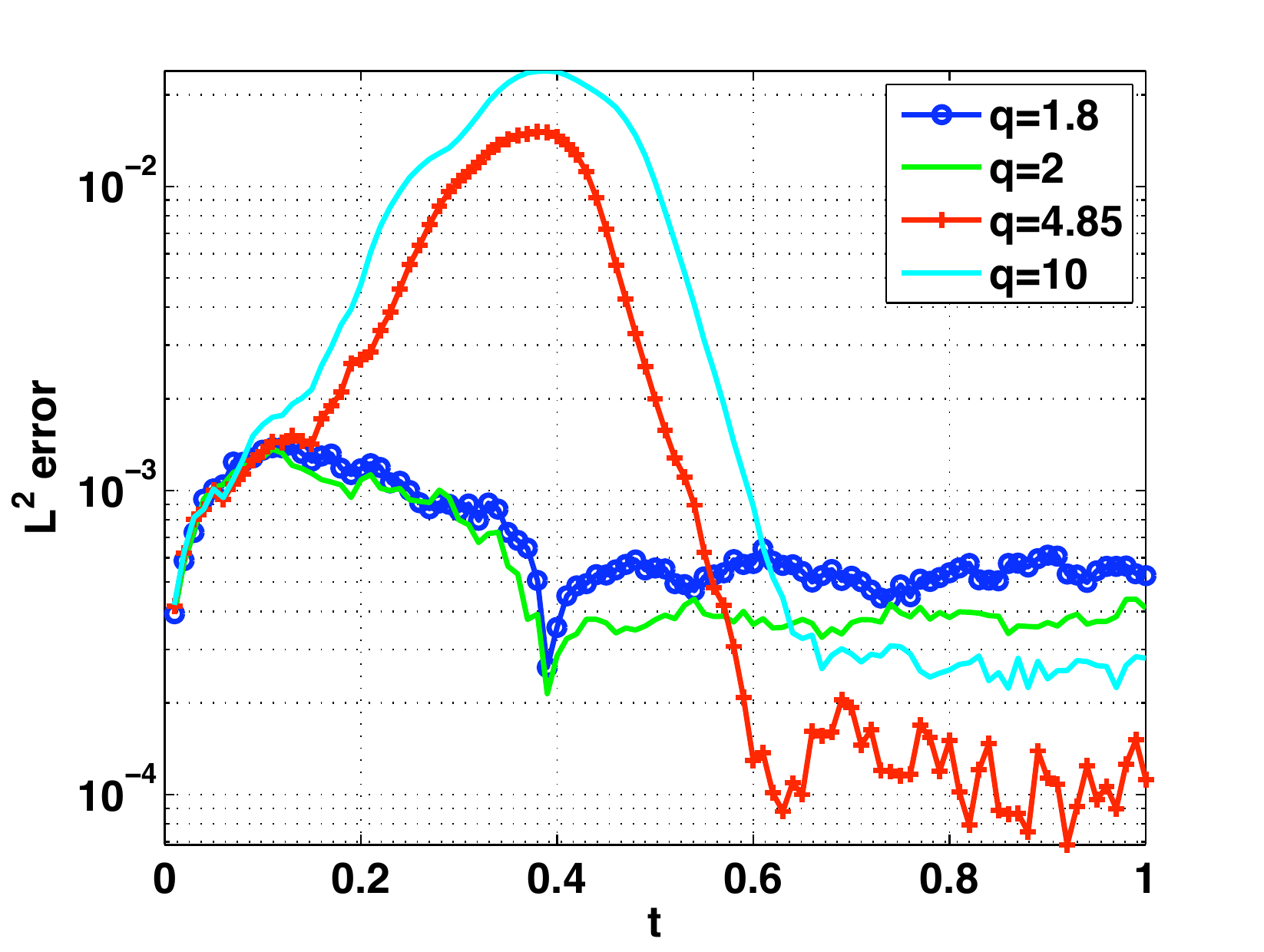}\\
    \caption{Sample realization with $p=10$ of the time evolution problem with
      a time-dependent stochastic forcing $\nu=0.2$. The time evolution is shown with
      equal steps of $\Dt=0.01$. In (a)-(c) we show the effect of an $H^1$
      noise in space, white in time.  In (d)-(f) we show the effect of a white noise in
      both space and time.}
    \label{fig:plaptimeNSE}
  \end{center}
\end{figure}

Let us now consider the stochastically forced case, $\nu>0$.
\Figref{fig:plaptimeNSE} shows plotted solution with time dependent
noises which are: $H^1$ in space and white in time for
(a)-(c), and  white both in space and time for (d)-(f). The forcing
corresponds to $\nu=0.2$. We expect that on average the noisy solution
is simply that of the deterministic system (which has a unique stable
solution). 
We see in (a)-(b) and (d)-(e) the time evolution
of the solution for one particular realization of the noise. This
should be compared with the deterministic case in
\figref{fig:plaptime2}-(d).

In \figref{fig:plaptimeNSE}-(c) and \figref{fig:plaptimeNSE}-(f) we plot the evolution of the
error for fixed time-step with $q\in\{1,2,\qopt,10\}$.  Similar to
the deterministic case, in the transient regime the error for $\qopt$
and $q=10$ is far higher than that for $q=2$ or $q=1.8$. 
Where the noise effects are large  (such as for the spatially
correlated noise in $H^1$ in \figref{fig:plaptimeNSE}-(c)), the optimal basis is no longer clear.
However where the effect is small (such as the white noise in \figref{fig:plaptimeNSE}-(f)), we
clearly see in the time dependent evolution about the steady state
that the $q=4.85$ outperforms the other bases (and in particular
$q=2$). 
For sandpile-type problems the introduction of a spatially
white noise is a natural choice. It is interesting to note the
small effect on the dynamics in this realization.

At present it is unclear whether the use of a $q$-sine basis 
with $q\not=2$ provides any real
computational advantage over the natural choice $q=2$ for the solution 
of \eqref{evo_equ}. There is clearly a computational overhead in obtaining 
the former for $q \neq 2$ that impacts significantly  on efficiency. 
However we have observed that the nonlinear problem \eqref{eq:discrete} 
is solved faster in the optimal basis.  This is certainly worth
further investigation. For example for $N=40$ we observe approximately 
a 20\% speed up in the nonlinear solve over the standard 2-sine basis. 
Thus, if solving many fixed $p$-problems, the corresponding optimal $q$-sine basis may 
not only be more accurate but also more efficient. 
In our current implementation the $q$-sine basis may be precomputed and stored.
Lemma~\ref{Sobolev_belongness}
give some prelimeinary indication on how to solve
the problem of apriori determining the optimal basis. Our numerical 
results suggest that for large $p$ we expect an optimal basis with $q>2$ 
for a right hand side with a discontinuous derivative.

\section*{Acknowledgements} We kindly thank Adrien Vignes and Bryan Rynne
for their thoughtful comments and involvement in discussions related to this
paper. The first author acknowledges support from the Universit{\'e} Paris Dauphine,
where part of this research was carried out.

\bibliographystyle{plainnat}
\bibliography{biblio_plap}

\begin{thebibliography}{23}
\providecommand{\natexlab}[1]{#1}
\providecommand{\url}[1]{\texttt{#1}}
\expandafter\ifx\csname urlstyle\endcsname\relax
  \providecommand{\doi}[1]{doi: #1}\else
  \providecommand{\doi}{doi: \begingroup \urlstyle{rm}\Url}\fi

\bibitem[Andreu et~al.(2009)Andreu, Maz{\'o}n, Rossi, and Toledo]{MR2481827}
F.~Andreu, J.~M. Maz{\'o}n, J.~D. Rossi, and J.~Toledo.
\newblock The limit as {$p\to\infty$} in a nonlocal {$p$}-{L}aplacian evolution
  equation: a nonlocal approximation of a model for sandpiles.
\newblock \emph{Calc. Var. Partial Differential Equations}, 35\penalty0
  (3):\penalty0 279--316, 2009.

\bibitem[Aronsson et~al.(1996)Aronsson, Evans, and Wu]{MR1419017}
G.~Aronsson, L.~C. Evans, and Y.~Wu.
\newblock Fast/slow diffusion and growing sandpiles.
\newblock \emph{J. Differential Equations}, 131\penalty0 (2):\penalty0
  304--335, 1996.

\bibitem[Barrett and Liu(1993{\natexlab{a}})]{MR1192966}
J.~W. Barrett and W.~B. Liu.
\newblock Finite element approximation of the {$p$}-{L}aplacian.
\newblock \emph{Math. Comp.}, 61\penalty0 (204):\penalty0 523--537,
  1993{\natexlab{a}}.

\bibitem[Barrett and Liu(1994)]{MR1276708}
J.~W. Barrett and W.~B. Liu.
\newblock Finite element approximation of the parabolic {$p$}-{L}aplacian.
\newblock \emph{SIAM J. Numer. Anal.}, 31\penalty0 (2):\penalty0 413--428,
  1994.

\bibitem[Barrett and Liu(1993{\natexlab{b}})]{ref_3}
J.~W. Barrett and W.~B. Liu.
\newblock Higher order regularity for the solution of some nonlinear degenerate
  elliptic equations.
\newblock \emph{SIAM J. Math. Anal.}, 24:\penalty0 1522 -- 1536,
  1993{\natexlab{b}}.

\bibitem[Barrett and Prigozhin(2000)]{Barrett2000977}
J.~W. Barrett and L.~Prigozhin.
\newblock Bean's critical-state model as the {$p\to\infty$} limit of an
  evolutionary {$p$}-{L}aplacian equation.
\newblock \emph{Nonlinear Analysis}, 42\penalty0 (6):\penalty0 977--993, 2000.

\bibitem[Bennewitz and Sait\={o}(2004)]{bensai}
C.~Bennewitz and Y.~Sait\={o}.
\newblock Approximation numbers of {S}obolev embedding operators on an
  interval.
\newblock \emph{J. London Math. Soc.}, 70:\penalty0 244--260, 2004.

\bibitem[Binding et~al.(2006)Binding, Boulton, {\v{C}}epi{\v{c}}ka, Dr{\'a}bek,
  and Girg]{MR2240660}
P.~Binding, L.~Boulton, J.~{\v{C}}epi{\v{c}}ka, P.~Dr{\'a}bek, and P.~Girg.
\newblock Basis properties of eigenfunctions of the {$p$}-{L}aplacian.
\newblock \emph{Proc. Amer. Math. Soc.}, 134\penalty0 (12):\penalty0 3487--3494
  (electronic), 2006.

\bibitem[Caboussat and Glowinski(2009)]{MR2449104}
A.~Caboussat and R.~Glowinski.
\newblock A numerical method for a non-smooth advection-diffusion problem
  arising in sand mechanics.
\newblock \emph{Commun. Pure Appl. Anal.}, 8\penalty0 (1):\penalty0 161--178,
  2009.

\bibitem[Da~Prato and Zabczyk(1992)]{DaPZ}
G.~Da~Prato and J.~Zabczyk.
\newblock \emph{Stochastic Equations in Infinite Dimensions}, volume~44 of
  \emph{Encyclopedia of Mathematics and its Applications}.
\newblock Cambridge University Press, Cambridge, 1992.

\bibitem[Ebmeyer and Liu(2005)]{ref1_2}
C.~Ebmeyer and W.~B. Liu.
\newblock Quasi-norm interpolation error estimates for the piecewise linear
  finite element approximation of p-{L}aplace equations.
\newblock \emph{Numer. Math.}, 100:\penalty0 233--258, 2005.

\bibitem[Ebmeyer et~al.(2005)Ebmeyer, Liu, and Steinhauer]{ref1_1}
C.~Ebmeyer, W.~B. Liu, and M.~Steinhauer.
\newblock Global regularity in fractional order {S}obolev spaces for the
  p-{L}aplace equation on polyhedral domains.
\newblock \emph{J. Anal. Appl.}, 24:\penalty0 353--237, 2005.

\bibitem[Elbert(1981)]{MR680591}
{\'A}.~Elbert.
\newblock A half-linear second order differential equation.
\newblock In \emph{Qualitative theory of differential equations, {V}ol. {I},
  {II} ({S}zeged, 1979)}, volume~30 of \emph{Colloq. Math. Soc. J\'anos
  Bolyai}, pages 153--180. North-Holland, Amsterdam, 1981.

\bibitem[Evans et~al.(1997)Evans, Feldman, and Gariepy]{MR1451539}
L.~C. Evans, M.~Feldman, and R.~F. Gariepy.
\newblock Fast/slow diffusion and collapsing sandpiles.
\newblock \emph{J. Differential Equations}, 137\penalty0 (1):\penalty0
  166--209, 1997.

\bibitem[Falcone and Finzi~Vita(2006)]{MR2240806}
M.~Falcone and S.~Finzi~Vita.
\newblock A finite-difference approximation of a two-layer system for growing
  sandpiles.
\newblock \emph{SIAM J. Sci. Comput.}, 28\penalty0 (3):\penalty0 1120--1132
  (electronic), 2006.

\bibitem[Igbida(2010)]{Igbida}
N.~Igbida.
\newblock A generalized collapsing sandpile model.
\newblock \emph{Archiv der Mathematik}, 94\penalty0 (2):\penalty0 193--200,
  2010.

\bibitem[Kuijper(2007)]{Kuijper1}
A.~Kuijper.
\newblock Image analysis using $p$-{L}aplacian and geometrical {PDE}s.
\newblock \emph{PAMM}, 7\penalty0 (1):\penalty0 1011201--1011202, 2007.

\bibitem[Lindqvist(1995)]{MR1469702}
P.~Lindqvist.
\newblock Some remarkable sine and cosine functions.
\newblock \emph{Ricerche Mat.}, 44\penalty0 (2):\penalty0 269--290, 1995.

\bibitem[Liu(2009)]{Liu}
W.~Liu.
\newblock On the stochastic $p$-{L}aplace equation.
\newblock \emph{J. of Mathematical Analysis and Applications}, 360:\penalty0
  737--751, 2009.

\bibitem[{\^O}tani(1984)]{MR753635}
M.~{\^O}tani.
\newblock A remark on certain nonlinear elliptic equations.
\newblock \emph{Proc. Fac. Sci. Tokai Univ.}, 19:\penalty0 23--28, 1984.

\bibitem[Pr\'ev\^ot and R\"ockner(2007)]{PrvtRcknr}
C.~Pr\'ev\^ot and M.~R\"ockner.
\newblock \emph{A Concise Course on Stochastic Partial Differential Equations}.
\newblock Springer, 2007.

\bibitem[Rossi(2010)]{Rossi_TOW}
J.~Rossi.
\newblock Tug-of-war games. {G}ames that {PDE} people like to play.
\newblock \emph{Preprint}, 2010.

\bibitem[Zhang et~al.(2007)Zhang, Peng, and Wu]{ZHANG2007546}
H-Y. Zhang, Q-C. Peng, and Y-D. Wu.
\newblock Wavelet inpainting based on {$p$}-{L}aplace operator.
\newblock \emph{Acta Automatica Sinica}, 33\penalty0 (5):\penalty0 546--549,
  2007.

\end{thebibliography}

\end{document}